\documentclass[12pt]{amsart}
\usepackage{amsmath,amsthm,amssymb,amscd,epsfig,latexsym,graphicx,eucal,layout,fancyhdr,bm,mathrsfs}
\usepackage{hyperref}

\newtheorem{theorem}{Theorem}[section]
\newtheorem{lemma}[theorem]{Lemma}

\newtheorem{corollary}[theorem]{Corollary}

\theoremstyle{definition}

\newtheorem{remark}[theorem]{\it Remark}

\newtheoremstyle{named}{}{}
{\itshape}{}{\bfseries}{.}{0.5em}
{#1 \thmnote{#3}}
\theoremstyle{named}
\newtheorem*{namedtheorem}{Theorem}

\newcommand{\con}{\operatorname{const.}}

\newcommand{\fix}{\operatorname{Fix}}
\newcommand{\isom}{\operatorname{Isom}^+(\HH^4)}
\newcommand{\Isom}{\operatorname{Isom}^+(\HH^n)}
\newcommand{\mob}{\operatorname{M\ddot{o}b}}
\newcommand{\mt}{M_{\operatorname{thin}}}
\newcommand{\bb}{\mathscr{B}}
\newcommand{\ff}{\mathscr{F}}
\newcommand{\bd}{\partial}

\newcommand{\es}{\emptyset}
\newcommand{\ov}{\overline}
\newcommand{\sm}{\smallsetminus}
\newcommand{\ve}{\varepsilon}

\newcommand{\vs}{\vspace{2mm}}

\newcommand{\RR}{{\mathbb R}}
\newcommand{\HH}{{\mathbb H}}
\newcommand{\TT}{{\mathbb T}}
\newcommand{\ZZ}{{\mathbb Z}}
\newcommand{\NN}{{\mathbb N}}

\newcommand{\ra}{R_{\alpha}}
\newcommand{\bit}{\it \bfseries}
\newcommand{\thmref}[1]{Theorem~\ref{#1}}

\newcommand{\lemref}[1]{Lemma~\ref{#1}}

\newcommand{\corref}[1]{Corollary~\ref{#1}}
\newcommand{\figref}[1]{Fig.~\ref{#1}}

\def\Empty{}
\newcommand\oplabel[1]{
  \def\OpArg{#1} \ifx \OpArg\Empty {} \else
    \label{#1}
  \fi}

\newcommand{\comm}[1]{}

\input{psbox}

\psfordvips

\begin{document}

\title{On Margulis cusps of hyperbolic $4$-manifolds}

\author[V. Erlandsson and S. Zakeri]{Viveka Erlandsson and Saeed Zakeri}

\address{V. Erlandsson, Department of Mathematics, Graduate Center of
CUNY, New York}

\email{verlandsson@gc.cuny.com}

\address{S. Zakeri, Department of Mathematics, Queens College and Graduate Center of
CUNY, New York}

\email{saeed.zakeri@qc.cuny.edu}

\date{\today}


\maketitle

\begin{abstract}
We study the geometry of the Margulis region associated with an irrational screw  translation $g$ acting on the $4$-dimensional real hyperbolic space. This is an invariant domain with the parabolic fixed point of $g$ on its boundary which plays the role of an invariant horoball for a translation in dimensions $\leq 3$. The boundary of the Margulis region is described in terms of a function $\bb_\alpha : [0,\infty) \to \RR$ which solely depends on the rotation angle $\alpha \in \RR/\ZZ$ of $g$. We obtain an asymptotically universal upper bound for $\bb_\alpha(r)$ as $r \to \infty$ for arbitrary irrational $\alpha$, as well as lower bounds when $\alpha$ is Diophatine and the optimal bound when $\alpha$ is of bounded type. We investigate the implications of these results for the geometry of Margulis cusps of hyperbolic $4$-manifolds that correspond to irrational screw translations acting on the universal cover. Among other things, we prove bi-Lipschitz rigidity of these cusps.    
\end{abstract}

\tableofcontents

\section{Introduction}
\label{sec:intro}

Let $\Isom$ denote the group of orientation-preserving isometries of the $n$-dimensional hyperbolic space. Consider a discrete subgroup $\Gamma \subset \Isom$ which acts freely on $\HH^n$, and suppose $p \in \bd \HH^n$ is a parabolic fixed point of $\Gamma$ with stabilizer $\Gamma_p \subset \Gamma$. A domain $X \subset \HH^n$ is said to be {\bit precisely invariant} under $\Gamma_p$ if $g(X)=X$ for all $g \in \Gamma_p$ and $g(X) \cap X=\emptyset$ for all $g \in \Gamma \sm \Gamma_p$. It is well known that in dimensions $n \leq 3$ one can always find a horoball based at $p$ that is precisely invariant under $\Gamma_p$. This allows a simple description of the corresponding cusp of the hyperbolic manifold $\HH^n/\Gamma$. In dimensions $n \geq 4$, however, examples constructed by Apanasov \cite{A1} and Ohtake \cite{Oh} show that such precisely invariant horoballs need not exist. The phenomenon is essentially due to the fact that in low dimensions the action of a parabolic isometry on $\bd \HH^n$ is conjugate to a translation, while in higher dimensions it is conjugate to a translation followed by a rotation. To distinguish the two types, we call the former a {\bit pure translation} and the latter a {\bit screw translation} (think of the motion of a Phillips screwdriver as you tighten a screw). A screw translation is {\bit rational} if the associated rotation has finite order, and {\bit irrational} otherwise. \vs

There is a standard way to construct precisely invariant domains that works in all dimensions, although it does not always produce horoballs. For any $\ve>0$ less than the Margulis constant of $\HH^n$, the {\bit Margulis region} $T(\Gamma_p)$ consisting of all points in $\HH^n$ that are moved a distance less than $\ve$ by some non-identity isometry in $\Gamma_p$ is precisely invariant under $\Gamma_p$. The corresponding {\bit Margulis cusp} $T(\Gamma_p)/\Gamma_p$ embeds isometrically into the quotient manifold $\HH^n/\Gamma$ and forms a component of the $\ve$-thin part in Thurston's thick-thin decomposition of $\HH^n/\Gamma$. When $\Gamma_p$ contains a pure or rational screw translation, $T(\Gamma_p)$ contains a horoball based at $p$, and this horoball is automatically precisely invariant under $\Gamma_p$. But when $\Gamma_p$ consists only of irrational screw translations, $T(\Gamma_p)$ cannot contain any horoball, and the examples of Apanasov and Ohtake show that $\Gamma_p$ may not have any precisely invariant horoball at all. In this case, the geometry of the Margulis region $T(\Gamma_p)$ becomes relevant in understanding the parabolic end of $\HH^n/\Gamma$ determined by the cusp $T(\Gamma_p)/\Gamma_p$. In this paper we study this geometry for $n=4$, the lowest dimension in which screw translations can exist. \vs

We use coordinates $(r, \theta, z, t)$ in the upper half-space model of $\HH^4$, where $(r, \theta, z)$ are the cylindrical coordinates of $\RR^3$, with $\theta \in \RR/\ZZ$, and $t>0$. After a suitable change of coordinates, we can put the parabolic fixed point at $p=\infty$. In the presence of an irrational screw translation, the stabilizer $\Gamma_{\infty}$ is necessarily cyclic (\thmref{cyclic}), so we may assume it is generated by the parabolic isometry
$$
g_\alpha: (r,\theta,z,t) \mapsto (r,\theta+\alpha,z+1,t) 
$$
which rotates by the angle $2 \pi \alpha$ around and translates a unit distance along the $z$-axis. In \cite{Su}, Susskind used this normalization to give the following explicit description of the Margulis region $T(\Gamma_{\infty})$ which now depends only on $\alpha$ and henceforth will be denoted by $T_\alpha$:
$$
T_\alpha = \{ (r, \theta, z, t) \in \HH^4 : t > \bb_\alpha(r) \}.
$$
Here the {\bit boundary function} $\bb_\alpha : [0,\infty) \to \RR$ is given by
$$
\bb_\alpha(r)= \inf_{j \geq 1} u_{\alpha,j}(r),
$$
where the sequence of functions $\{ u_{\alpha,j} \}_{j \geq 1}$ is defined by
$$
u_{\alpha,j}(r) = \con \, \sqrt{4 \sin^2(\pi j\alpha) \ r^2+j^2}
$$
(see \S \ref{sec:p}). The key step in understanding the behavior of the boundary function is to decide which $u_{\alpha,j}$'s are the constituents of $\bb_\alpha$, that is, for which indices $j$ we have $u_{\alpha,j}=\bb_\alpha$ in some non-empty open interval. When this happens, we say that $j$ is {\bit present}; otherwise we say it is {\bit absent}. The decision between presence and absence of a given index depends on the arithmetical properties of the rotation angle $\alpha$. Expand $\alpha$ into its continued fraction $[a_1, a_2, a_3, \ldots]$, where $a_n \in \NN$, and let $p_n/q_n = [a_1, a_2,\ldots, a_n]$ be the $n$-th rational convergent of $\alpha$. Using the fact that the denominators $\{ q_n \}$ form the moments of closest return of every orbit of the irrational rotation $\theta \mapsto \theta+\alpha$ of the circle, it is not hard to prove that all the present indices must be of the form $q_n$ for some $n$ (see \cite{Su}, and \lemref{paq} below). Simple examples show that not all the $q_n$ are necessarily present (compare \figref{TO}). But a combinatorial analysis of the functions $u_{\alpha,j}$ that we carry out in \S \ref{sec:ca} proves that no two consecutive elements of the sequence $\{ q_n \}$ can be absent (\thmref{p/a}). This, in turn, leads to a combinatorial characterization of presence (\corref{sufi}) which is further developed into a purely arithmetical characterization in the Appendix. \vs

In \S \ref{sec:aa} these results team up with detailed estimates from continued fraction theory to yield the following

\begin{namedtheorem}[A]\label{bb}
For every irrational $\alpha \in \RR/\ZZ$, the boundary function $\bb_\alpha$ satisfies the asymptotically universal upper bound
$$
\bb_\alpha(r) \leq \con \, \sqrt{r} \qquad \qquad \text{for large} \ r.
$$
If $\alpha$ is Diophantine of exponent $\nu \geq 2$, then $\bb_\alpha$ satisfies the lower bound
$$
\bb_\alpha(r) \geq \con  \, r^{1/(2\nu-2)} \qquad \text{for large} \ r.
$$
\end{namedtheorem}

\noindent
The upper bound is asymptotically universal in the sense that the constant involved is independent of $\alpha$. In fact, a quantitative version of this result (\corref{expl}) shows that
$$
\sup_{\alpha} \ \limsup_{r \to \infty} \frac{\bb_\alpha(r)}{\sqrt{r}} \leq 1000.
$$
On the other hand, there are irrational numbers $\alpha$ of Liouville type for which $\bb_{\alpha}$ has arbitrarily slow growth over long intervals (\thmref{lio}). Furthermore, the estimates leading to Theorem A allow us to prove that $\alpha$ being of bounded type is the optimal condition for $\bb_\alpha(r)$ to grow asymptotically like $\sqrt{r}$ (\thmref{D2}). This is a sharpened version of the main result of \cite{E} which carried out a similar program for bounded type irrationals. \vs

In \S \ref{sec:geo} we study the Margulis cusps $C_\alpha = T_\alpha / \langle g_\alpha \rangle$ for irrational $\alpha$. These cusps are topologically indistinguishable since they are all homeomorphic to the product $\RR^3 \times S^1$. In fact, for any pair of irrationals $\alpha,\beta$, we can find a piecewise-smooth homeomorphism $T_\alpha \to T_\beta$ which conjugates $g_\alpha$ to $g_\beta$ (compare formula \eqref{psh}). This is in stark contrast to the situation in low-dimensional dynamics where the rotation angle is a topological invariant. On the other hand, it is readily seen that these cusps are never isometric to one another, for any isometry $C_\alpha \to C_\beta$ lifts to an element of $\isom$ which conjugates $g_\alpha$ to $g_\beta$ on their respective Margulis regions, hence everywhere in $\HH^4$, implying $\alpha=\beta$. That the geometry of $C_\alpha$ determines the rotation angle $\alpha$ uniquely has an alternative explanation that we outline in \S \ref{sec:geo} by showing that the boundary function $\bb_\alpha$ (and therefore $\alpha$) can be recovered from the volume of leaves in a canonical $3$-dimensional foliation of $C_\alpha$. \vs

In \S \ref{sec:geo} we prove a stronger form of rigidity for Margulis cusps:

\begin{namedtheorem}[B]\label{rig}
Suppose there is a bi-Lipschitz embedding $C_\alpha \hookrightarrow C_\beta$ for some irrationals $\alpha,\beta$. Then $\alpha=\beta$.
\end{namedtheorem}

\noindent
This follows from the corresponding statement on the universal covers which asserts that if $\varphi: T_\alpha \hookrightarrow \HH^4$ is a bi-Lipschitz embedding which satisfies $\varphi \circ g_{\alpha}=g_{\beta} \circ \varphi$, then $\alpha=\beta$ (\thmref{a=b}). The proof of this result is based on comparing the return maps of carefully chosen iterates of $g_\alpha, g_\beta$ on certain subvarieties of $\HH^4$ defined by the functions $u_{\alpha,j}$ as one goes out to $\infty$. As a special case, we recover a result of Kim in \cite{Ki} on global conjugacies between irrational screw translations (\corref{kim}). \vs

Our study of the Margulis region provides an example of the rich and non-trivial phenomena that even the simplest infinite volume hyperbolic manifolds can exhibit. The results presented here have analogs in dimensions $>4$ which are currently being investigated by us and will be the subject of a sequel paper. To demonstrate the application of these ideas in related problems, in \cite{EZ} we use boundary functions of Margulis regions to prove a discreteness criterion for subgroups of $\Isom$ which contain a parabolic isometry. This can be viewed as a generalization of the well-known results of Shimizu-Leutbecher \cite{Sh} and J{\o}rgensen \cite{J} in dimensions $2$ and $3$, and gives a sharper asymptotic bound than the previously known results such as Waterman's inequality in \cite{W}. \vs \vs

\noindent
{\it Acknowledgements}. We are grateful to Ara Basmajian for sharing his knowledge and lending his support at various stages of this project. We also thank Perry Susskind for useful conversations on the topics discussed here.

\section{Preliminaries}
\label{sec:p}

Much of the following is valid for hyperbolic spaces of arbitrary dimension, but for simplicity we focus on the $4$-dimensional case only. Further details on the subject can be found in \cite{A2}, \cite{BP}, \cite{R}, and \cite{Th}.

\subsection*{Isometries of $\HH^4$}

We will use the upper half-space model for the hyperbolic space $\HH^4$:
$$
\HH^4 = \{ x = (v,t) : v \in \RR^3, t>0 \} \subset \RR^4.
$$
The extended boundary $\bd \HH^4 = \RR^3 \cup \{ \infty \}$ is homeomorphic to the $3$-sphere, with the closure $\ov{\HH} \ \! ^4= \HH^4 \cup \bd \HH^4$ homeomorphic to the closed $4$-ball. The hyperbolic metric $dx^2/t^2$ on $\HH^4$ induces the distance $\rho(\cdot,\cdot)$ which satisfies
\begin{equation}\label{dist}
\cosh(\rho(x,\hat{x}))=1+\frac{\| x-\hat{x} \|^2}{2t \hat{t}} \qquad (x,\hat{x} \in \HH^4).
\end{equation}
Here $\| \cdot \|$ is the Euclidean norm in $\RR^4$ and $t,\hat{t}$ are the last coordinates of $x,\hat{x}$. We denote by $\isom$ the group of orientation-preserving isometries of $\HH^4$ with respect to the hyperbolic metric. It is well known that every element of $\isom$ extends continuously to a M\"{o}bius map acting on $\bd \HH^4$. Conversely, the Poincar\'e extension of every M\"{o}bius map of $\bd \HH^4$ is an element of $\isom$. It follows that $\isom$ is canonically isomorphic to the group $\mob(3)$ of orientation-preserving M\"{o}bius maps acting on the $3$-sphere. For each $\gamma \in \isom$, the fixed point set $\fix(\gamma)=\{ x \in \ov{\HH} \ \! ^4 : \gamma(x)=x \}$ is non-empty. A non-identity $\gamma$ is {\bit elliptic} if $\fix(\gamma)$ intersects $\HH^4$, {\bit loxodromic} if $\fix(\gamma)$ consists of two distinct points on $\bd \HH^4$, and {\bit parabolic} if $\fix(\gamma)$ consists of a unique point on $\bd \HH^4$. The three cases exhaust all possibilities. Elliptic isometries will not be discussed in this paper. \vs

Every loxodromic $\gamma \in \isom$ is conjugate to the normal form
\begin{equation}\label{normalform0}
(v,t) \mapsto (\lambda Av, \lambda t) \qquad (v \in \RR^3,t>0)
\end{equation}
fixing $0, \infty \in \bd \HH^4$, where $A \in \text{SO}(3)$ and $\lambda >1$. The number $\lambda$ and the conjugacy class of $A$ in $\text{SO}(3)$ are uniquely determined by $\gamma$. The map \eqref{normalform0} acts as hyperbolic translation by $\ell(\gamma)=\log \lambda >0$ on the vertical geodesic $\{ (0,t) , t >0 \}$ joining $0$ and $\infty$. It follows that $\gamma$ acts as hyperbolic translation by $\ell(\gamma)$ on the geodesic which joins its pair of fixed points. We call this geodesic the {\bit axis} of $\gamma$. \vs

Every parabolic $\gamma \in \isom$ is conjugate to the normal form
\begin{equation}\label{normalform1}
(v,t) \mapsto (Av+a,t) \qquad (v \in \RR^3,t>0)
\end{equation}
fixing $\infty \in \bd \HH^4$, where $A \in \text{SO}(3)$ and $a \in \RR^3$ is non-zero. The conjugacy class of $A$ in $\text{SO}(3)$ is uniquely determined by $\gamma$. If $A \neq I$, there is a decomposition $\RR^3 = E \oplus E^{\perp}$ into the $1$-dimensional subspace $E$ on which $A$ acts as the identity, and its orthogonal complement $E^{\perp}$ on which $A$ acts as a rotation. Thus, $E$ can be thought of as the axis of this rotation. The map $A-I$ has kernel $E$ and maps $E^{\perp}$ isomorphically onto itself. Since the map \eqref{normalform1} has no fixed point in $\RR^3$, it follows that $a \notin E^{\perp}$. After a further conjugation by a Euclidean translation, we may therefore arrange $a \in E$, or $Aa=a$. We call $\gamma$ a {\bit pure translation} if $A=I$, and a {\bit screw translation} if $A \neq I$. A screw translation is {\bit rational} if $A$ has finite order, and is {\bit irrational} otherwise. These correspond to the cases where the angle of rotation of $A$ about its axis is a rational or irrational multiple of $2\pi$. \vs

Commuting elements of $\isom$ are closely linked: If $\gamma,\eta$ are non-identity and non-elliptic with $\gamma \eta= \eta \gamma$, then $\fix(\gamma)=\fix(\eta)$. It follows that $\gamma,\eta$ are either loxodromics with a common axis, or parabolics with a common fixed point.

\subsection*{The Margulis region}

Fix some $\ve>0$. For every $\gamma \in \isom$ consider the open set
$$
T_\gamma = \{ x \in \HH^4: \rho(\gamma(x),x)<\ve \}.
$$
Convexity of the function $x \mapsto \rho(\gamma(x),x)$ on $\HH^4$ \cite[Theorem 2.5.8]{Th} shows that $T_\gamma$ is always convex. The relation
\begin{equation}\label{inv}
g(T_\gamma) = T_{g \gamma g^{-1}} \qquad \text{for all} \ \gamma, g \in \isom
\end{equation}
follows immediately from the definition. \vs

Suppose $\gamma$ is loxodromic of the form \eqref{normalform0} with fixed points at $0,\infty$. By \eqref{dist}, if $x=(v,t) \in \HH^4$,
$$
\cosh(\rho(\gamma(x),x))=1+\frac{\| \lambda Av-v \|^2 + (\lambda-1)^2 t^2}{2 \lambda t^2} \geq 1+ \frac{(\lambda-1)^2}{2\lambda}= \cosh( \log \lambda ).
$$
It follows that $T_\gamma = \es$ if $\ell(\gamma) \geq \ve$. On the other hand, if $\ell(\gamma) < \ve$, the same formula applied to $x=(0,t)$ gives
$$
\cosh(\rho(\gamma(x),x))=1+\frac{(\lambda-1)^2 t^2}{2 \lambda t^2} = \cosh( \log \lambda) < \cosh (\ve),
$$
so $x \in T_\gamma$. This shows that $T_\gamma$ is an open neighborhood of the axis of $\gamma$ (more precisely, it is an equidistant neighborhood of the axis of $\gamma$, but we do not need this property). \vs

Now suppose $\gamma$ is parabolic fixing $\infty$, so it has the normal form \eqref{normalform1}. By \eqref{dist}, if $x=(v,t) \in \HH^4$,
$$
\cosh(\rho(\gamma(x),x))=1+\frac{\| (A-I)v+a \|^2}{2t^2},
$$
which shows $x \in T_\gamma$ for each fixed $v$ if $t$ is sufficiently large. Note that in the case of a pure translation $(v,t) \mapsto (v+a,t)$, the above formula shows that $x=(v,t) \in T_\gamma$ if and only if $t> \| a \| / \sqrt{2\cosh(\ve)-2}$, which describes a horoball based at $\infty$. \vs

These observations, combined with \eqref{inv}, prove the following

\begin{lemma}\label{tog}
\mbox{}
\begin{enumerate}
\item[(i)]
If $\gamma$ is loxodromic, then $T_\gamma$ is a convex neighborhood of the axis of $\gamma$ when $\ell(\gamma) < \ve$, and $T_\gamma = \es$ when $\ell(\gamma) \geq \ve$. In particular, if $\gamma, \eta$ are loxodromics with a common axis and $\ell(\gamma), \ell(\eta)$ are both $<\ve$, then $T_\gamma \cap T_\eta \neq \es$. \vs
\item[(ii)]
If $\gamma$ is parabolic, then $T_\gamma$ is a convex domain having the fixed point of $\gamma$ on its boundary. Every geodesic landing at this fixed point eventually enters $T_\gamma$. In particular, if $\gamma, \eta$ are parabolics with a common fixed point, then $T_\gamma \cap T_\eta \neq \es$.
\end{enumerate}
\end{lemma}

Let $\Gamma$ be a discrete subgroup of $\isom$ acting freely on $\HH^4$ (thus, there are no elliptics in $\Gamma$), and let $M$ denote the quotient manifold $\HH^4/\Gamma$. The {\bit thin part} of $M$, denoted by $\mt$, is the set of points in $M$ at which the injectivity radius is $< \ve/2$. Equivalently, $\mt$ is the set of points in $M$ through which a geodesic loop of hyperbolic length $< \ve$ passes. Under the canonical projection $\HH^4 \to M$, the thin part lifts to the union
\begin{equation}\label{tg}
T(\Gamma) = \bigcup_{\gamma \in \Gamma \sm \{ \text{id} \}} T_\gamma.
\end{equation}
The structure of the set $T(\Gamma)$, hence $\mt$, can be understood when $\ve$ is sufficiently small. The key tool is the following special case of a result due to Zassenhaus and Kazhdan-Margulis, often known as the ``Margulis Lemma:'' There is a universal constant $\ve_4>0$ (called the {\bit Margulis constant} of $\HH^4$) such that if $0< \ve \leq \ve_4$ and $x \in \HH^4$, the group $\Gamma^x$ generated by $\{ \gamma \in \Gamma : x \in T_\gamma \}$ is {\bit virtually abelian} in the sense that it has an abelian subgroup of finite index \cite{BP}. In particular, any two elements of $\Gamma^x$ have powers that commute with one another. \vs

The following is a converse to \lemref{tog}:

\begin{lemma}\label{tng}
Let $0 < \ve \leq \ve_4$. Suppose $T_\gamma \cap T_\eta \neq \es$ for some $\gamma,\eta \in \Gamma \sm \{ \operatorname{id} \}$. Then $\gamma,\eta$ are either loxodromics having a common axis, or parabolics having a common fixed point.
\end{lemma}

\begin{proof}
Take an $x \in T_\gamma \cap T_\eta$. Since $\gamma, \eta$ belong to the virtually abelian group $\Gamma^x$, some positive powers $\gamma^n$ and $\eta^m$ must commute. This implies $\text{Fix}(\gamma)=\text{Fix}(\gamma^n)=\text{Fix}(\eta^m)=\text{Fix}(\eta)$, from which the result follows immediately.
\end{proof}

From now on we fix an $\ve$ such that $0<\ve \leq \ve_4$. It follows from \lemref{tog} and \lemref{tng} that each connected component of $T(\Gamma)$ in \eqref{tg} is a union of the $T_\gamma$, where $\gamma$ runs over all loxodromics in $\Gamma$ having a common axis, or all parabolics in $\Gamma$ having a common fixed point. The component is called loxodromic or parabolic accordingly. In this paper, we are only interested in parabolic components of $T(\Gamma)$. \vs

If $p \in \bd \HH^4$ is a parabolic fixed point of $\Gamma$, the stabilizer
$$
\Gamma_p = \{ \gamma \in \Gamma: \gamma(p)=p \}
$$
is a maximal parabolic subgroup of $\Gamma$ (it contains no loxodromic since a loxodromic and a parabolic element sharing a fixed point would generate a non-discrete group \cite[Lemma D.3.6]{BP}). By the preceding remarks, the domain
\begin{equation}\label{ma}
T(\Gamma_p) = \bigcup_{\gamma \in \Gamma_p \sm \{ \text{id} \}} T_\gamma
\end{equation}
is a connected component of $T(\Gamma)$. We call $T(\Gamma_p)$ the {\bit Margulis region} associated with the parabolic fixed point $p$.

\begin{lemma}\label{mar}
The Margulis region $T=T(\Gamma_p)$ is precisely invariant under the action of $\Gamma_p$, in the sense that 
$$
\begin{cases}
g(T)=T & \qquad \text{if} \ \ g \in \Gamma_p \\
g(T) \cap T =\es & \qquad \text{if} \ \ g \in \Gamma \sm \Gamma_p.
\end{cases}
$$
\end{lemma}

\begin{proof}
If $\gamma \in \Gamma_p \sm \{ \text{id} \}$ and $g \in \Gamma_p$, \eqref{inv} shows that $g(T_\gamma) \subset T$. Taking the union over all such $\gamma$ gives $g(T) \subset T$. The same argument applied to $g^{-1}$ then shows $g(T)=T$. If $\gamma \in \Gamma_p \sm \{ \text{id} \}$ and $g \notin \Gamma_p$, then $g \gamma g^{-1} \notin \Gamma_p$, so by \eqref{inv} and \lemref{tng}, $g(T_\gamma) \cap T = \es$. Taking the union over all such $\gamma$ then proves $g(T) \cap T = \es$.
\end{proof}

The precise invariance of the Margulis region shows that the quotient $T(\Gamma_p)/\Gamma_p$ embeds isometrically into the hyperbolic manifold $M=\HH^4/\Gamma$ and forms a connected component of $\mt$. We call $T(\Gamma_p)/\Gamma_p$ the {\bit Margulis cusp} of $M$ associated with $p$ (more precisely, associated with the conjugacy class of $\Gamma_p$ in $\Gamma$).

\subsection*{Parabolic stabilizers}

We continue assuming that $\Gamma \subset \isom$ is a discrete group acting freely on $\HH^4$ and $p$ is a parabolic fixed point of $\Gamma$. We wish to show that the stabilizer subgroup $\Gamma_p$ has a simple algebraic structure in the presence of an irrational screw translation. Without loss of generality we may assume $p=\infty$. By the normal form \eqref{normalform1}, $\Gamma_\infty$ is isomorphic to a discrete subgroup of the group of orientation-preserving Euclidean isometries $v \mapsto Av+a$ of $\RR^3$. According to a classical theorem of Bieberbach, any such group is virtually abelian \cite[Theorem 4.2.2]{Th}.

\begin{lemma}\label{2,2}
Suppose $\gamma: (v,t) \mapsto (Av+a,t)$ and $\eta: (v,t) \mapsto (Bv+b,t)$ are commuting parabolics in $\isom$, with $Aa=a$. Then $AB=BA$ and $Ab=b$. If $A,B \neq I$, then $A,B$ have a common axis of rotation.
\end{lemma}

\begin{proof}
There is nothing to prove if $A=I$, so let us assume $A \neq I$. The condition $\gamma \eta = \eta \gamma$ implies
$$
AB=BA \qquad \text{and} \qquad Ab+a=Ba+b.
$$
If $B=I$, this gives $Ab=b$. If $B \neq I$, the condition $AB=BA$ shows that $B$ maps the axis $E$ of $A$ isomorphically onto itself. Since $B \in \text{SO}(3)$,  either $B|_E=I$ or $B|_E=-I$. In the first case, $E$ is the axis of $B$ also and $Ba=a$, so $Ab=b$. The second case is impossible since it would give $Ba=-a$ or $Ab-b=-2a$. Since $Ab-b \in E^{\perp}$ and $a \in E$, this would imply $a=0$.
\end{proof}

\begin{theorem}\label{cyclic}
If the stabilizer subgroup $\Gamma_\infty$ contains an irrational screw translation, it must be an infinite cyclic group.
\end{theorem}

\begin{proof}
First we show that $\Gamma_\infty$ contains no pure translations (hence no rational screw translations). Otherwise, since $\Gamma_\infty$ is virtually abelian by Bieberbach's theorem, we can find an irrational screw translation $\gamma: (v,t) \mapsto (Av+a,t)$, with $Aa=a$, which commutes with a pure translation $\eta: (v,t) \mapsto (v+b,t)$ in $\Gamma_\infty$. By \lemref{2,2}, both $a,b$ belong to the axis $E$ of $A$. The group $\langle \gamma , \eta \rangle \subset \Gamma_\infty$ is discrete and preserves $E$, so its action on $E$ must be generated by a single translation $v \mapsto v+v_0$. It follows that $a=mv_0$ and $b=nv_0$ for some non-zero integers $m,n$. This implies that the map $\gamma^n \eta^{-m}: (v,t) \mapsto (A^n v , t)$ in $\Gamma_\infty$ is elliptic, which is a contradiction. \vs

Thus, all non-identity elements of $\Gamma_\infty$ are irrational screw translations. Take two such elements $\gamma: (v,t) \mapsto (Av+a,t)$, with $Aa=a$, and $\eta : (v,t) \mapsto (Bv+b,t)$. There are powers
\begin{align*}
\gamma^n : (v,t) & \mapsto  (A^n v + na,t) \\
\eta^m : (v,t) & \mapsto (B^m v + c, t), \quad c=\sum_{k=0}^{m-1} B^k b
\end{align*}
that commute with each other. By \lemref{2,2}, $A^n$ and $B^m$ have a common axis $E$, hence the same is true of $A$ and $B$, so $AB=BA$. Moreover, $A^n c=c$, so $c \in E$, so $Ac=c$. Let $w=Ab-b \in E^{\perp}$. We have
$$
\sum_{k=0}^{m-1} B^k w = \sum_{k=0}^{m-1} B^k A b - \sum_{k=0}^{m-1} B^k b = \sum_{k=0}^{m-1} A B^k b - \sum_{k=0}^{m-1} B^k b= Ac-c=0.
$$
Applying $B$ on each side gives $\sum_{k=1}^m B^k w =0$. Subtracting the two equations, we obtain $B^m w - w=0$. Since $B^m: E^{\perp} \to E^{\perp}$ is an irrational rotation, it follows that $w=0$, or $Ab=b$. \vs

This shows that the non-identity elements of $\Gamma_\infty$ have a common axis of rotation and translation direction $E$. Let $G$ denote the restriction of $\Gamma_\infty$ to $E$. As a discrete group of translations of $E$, $G$ must be cyclic. The natural homomorphism $\Gamma_\infty \to G$ is injective since any element in its kernel must fix $E$ pointwise. Since $\Gamma_\infty$ is a parabolic group, such an element can only be the identity map. It follows that $\Gamma_\infty$ is isomorphic to the cyclic group $G$.
\end{proof}

For a different proof of \thmref{cyclic}, under the additional assumption that $\Gamma_\infty$ is abelian, see \cite{Ki}.

\subsection*{Explicit description of the Margulis region}

\thmref{cyclic} allows an explicit description of the Margulis region $T(\Gamma_\infty)$ in the presence of irrational screw translations. Let us use the coordinates $(r,\theta,z,t)$ in $\HH^4$, where $(r,\theta,z)$ are the cylindrical coordinates of $\RR^3$ and $t>0$. We assume $\theta \in \RR/\ZZ$, which amounts to measuring the polar angle in full turns rather than multiples of $2\pi$. In these coordinates, $\Gamma_\infty$ is conjugate to the group generated by the screw translation
\begin{equation}\label{ga}
g_{\alpha}:(r,\theta,z,t) \mapsto (r,\theta+\alpha,z+1,t)
\end{equation}
for a unique irrational $\alpha \in \RR/\ZZ$ called the {\bit rotation angle} of $\Gamma_\infty$. As the Margulis region $T(\Gamma_\infty)$ is now independent of the rest of the group $\Gamma$, and to emphasize its sole dependence on $\alpha$, we simplify the notation $T(\Gamma_\infty)$ to $T_\alpha$ and $T_{g_{\alpha}^j}$ to $T_{\alpha,j}$. Thus, \eqref{ma} takes the form
\begin{equation}\label{oa}
T_\alpha =  \bigcup_{j=1}^{\infty} T_{\alpha,j}.
\end{equation}
Observe that the union is taken over positive integers only since $T_{\alpha,j}=T_{\alpha,-j}$. \vs

By \eqref{dist}, the condition $x=(r,\theta,z,t) \in T_{\alpha,j}$ is equivalent to
$$
\frac{\| (r,\theta + j \alpha, z+j,t)-(r,\theta,z,t) \|^2}{2t^2} < \cosh(\ve)-1
$$
or
$$
\frac{4 \sin^2(\pi j\alpha) \ r^2+j^2}{2t^2} < \cosh(\ve)-1.
$$
Setting
\begin{equation}\label{UJ}
u_{\alpha,j}(r) = c(\ve) \, \sqrt{4 \sin^2(\pi j\alpha) \ r^2+j^2},
\end{equation}
where $c(\ve) =1/\sqrt{2\cosh(\ve)-2}$, this condition can be written as
$$
x=(r,\theta,z,t) \in T_{\alpha,j} \quad \Longleftrightarrow \quad t > u_{\alpha,j}(r).
$$
If we define the {\bit boundary function} $\bb_\alpha : [0,\infty) \to \RR$ by
\begin{equation}\label{psi}
\bb_\alpha(r)= \inf_{j \geq 1} u_{\alpha,j}(r),
\end{equation}
it follows that
$$
T_\alpha = \{ (r,\theta,z,t) \in \HH^4:  t > \bb_\alpha(r) \}.
$$

\subsection*{Continued fractions and rotations of the circle}
Our analysis of the boundary function of the Margulis region will depend on the continued fraction algorithm. Below we outline a few basic facts that are used in the next section. For a full treatment, see for example \cite{HW} or \cite{H}. \vs

Let $\TT = \RR/\ZZ$. Fix an irrational $\alpha \in \TT$ which may be identified with its unique representative in the interval $(0,1) \subset \RR$. Expand $\alpha$ into a continued fraction
$$
\alpha = \dfrac{1}{a_1+\dfrac{1}{a_2+\dfrac{1}{a_3+\dfrac{1}{\ddots}}}} = [a_1,a_2,a_3,\ldots],
$$
where the {\bit partial quotients} $a_n \in \NN$ are uniquely determined by $\alpha$. The truncated continued fractions
$$
\frac{p_n}{q_n} = [a_1,a_2,\ldots,a_n] \qquad (n \geq 1)
$$
are called the {\bit rational convergents} of $\alpha$. Setting $p_0=0, q_0=1$, they satisfy the recursions
\begin{equation}\label{pqr}
\begin{cases}
p_n = a_n \, p_{n-1} + p_{n-2} \\
q_n \, = a_n \, q_{n-1} + q_{n-2}
\end{cases}
\qquad (n \geq 2).
\end{equation}
The sequences $\{ p_n \}$ and $\{ q_n \}$ are increasing and tend to infinity
at least exponentially fast since $p_n \geq p_{n-1}+p_{n-2} > 2 p_{n-2}$ and similarly $q_n > 2q_{n-2}$. Note also that by the second recursion the knowledge of $\{ q_n \}$ will determine $\{ a_n \}$, hence $\alpha$, uniquely. \vs

The asymptotic behavior of the denominators $\{ q_n \}$ characterizes some important arithmetical classes of irrational numbers. For example, let ${\mathscr D}_{\nu}$ be the set of {\bit Diophantine numbers of exponent} $\nu \geq 2$:
$$
{\mathscr D}_{\nu} = \left\{ \alpha \in \TT : \left| \alpha - \frac{p}{q} \right| > \frac{\con}{|q|^{\nu}} \ \text{for every rational number} \ \frac{p}{q} \right\}.
$$
It is not hard to show that
$$
\alpha \in {\mathscr D}_{\nu} \Longleftrightarrow \sup _{n \geq 1} \ \frac{q_{n+1}}{q_n^{\, \nu-1}} < \infty.
$$
In particular, since $a_n < q_n/q_{n-1} < a_n+1$, we have
$$
\alpha \in {\mathscr D}_2 \Longleftrightarrow \sup_{n \geq 1} \ a_n < \infty.
$$
Because of this, Diophantine numbers of exponent $2$ are said to be of {\bit bounded type}. It is well known that ${\mathscr D}_{\nu}$ has full Lebesgue measure in $\TT$ if $\nu>2$ and zero measure if $\nu=2$. \vs

We now turn to rotations of the circle, represented as the additive group $\TT$. We equip $\TT$ with the ``norm''
$$
\| w \| = \min \{ |w-p| : p \in \ZZ \}
$$
which can be thought of as the distance from $w$ to $0$ and satisfies $\| w \| \leq 1/2$. The norm $\| w - \hat{w} \|$ serves as a natural distance between $w,\hat{w} \in \TT$. \vs

Each irrational number $\alpha \in \TT$ induces the rotation $\ra: \TT \to \TT$ defined by
$$
\ra: w \mapsto w+\alpha \qquad (\text{mod} \ \ZZ).
$$
The orbits of irrational rotations are dense in $\TT$, so the orbit $w_j= w+j \alpha \ (\text{mod} \ \ZZ)$ must return to any neighborhood of $w_0=w$ infinitely often. Let us say that an integer $q>0$ is a {\bit closest return moment} of the orbit of $w$ if
$$
\| w_q - w_0  \| = \| q \alpha \| < \| j \alpha \| = \| w_j - w_0 \| \qquad \text{whenever} \ \ 0<j<q.
$$
This simply means that $w_q$ is closer to $w_0$ than any of its predecessors in the orbit. Notice that this notion is independent of the choice of the initial point $w_0$.

\begin{theorem}[Dynamical characterization of continued fractions] \label{mom}
For every irrational number $\alpha \in \TT$, the denominators $\{ q_n \}$ of the rational convergents of $\alpha$ constitute the
closest return moments of the orbits of the rotation $\ra$.
\end{theorem}

We will make repeated use of the following properties of the norms $\| q_n \alpha \|$. As before, $\alpha \in \TT$ is an irrational number with partial quotients $\{ a_n \}$ and rational convergents $\{ p_n/q_n \}$.

\begin{lemma}\label{qnp}
For every $n\geq 1$, \vs
\begin{enumerate}
\item[(i)]
$\| q_n \alpha \| = |q_n \alpha - p_n| = (-1)^n (q_n \alpha - p_n)$. \vs
\item[(ii)]
$\dfrac{1}{2q_{n+1}} < \| q_n \alpha \| < \dfrac{1}{q_{n+1}}$. \vs
\item[(iii)]
$\| q_n \alpha \| = a_{n+2} \| q_{n+1} \alpha \| + \| q_{n+2} \alpha \|$.
\end{enumerate}
\end{lemma}

\begin{remark}
The case $n=0$ of the above lemma requires special care. If $0<\alpha<1/2$ so $a_1=q_1>q_0=1$, then $\| q_0 \alpha \| = \alpha$ and all three parts of the lemma remain true. However, if $1/2<\alpha<1$ so $a_1=q_1=q_0=1$, then $\| q_0 \alpha \| = 1-\alpha$ and one should modify the lemma by everywhere replacing $\| q_0 \alpha \|$ with $1-\| q_0 \alpha \|$.
\end{remark}

\section{Combinatorial analysis of the boundary function}
\label{sec:ca}

We begin our study of the boundary function $\bb_\alpha$ for a given irrational rotation angle $\alpha \in \TT$ with the partial quotients $\{ a_n \}$ and rational convergents $\{ p_n/q_n \}$. As the rotation angle is fixed throughout this section, we will drop the subscript $\alpha$ from our notations. Thus, the boundary function is $\bb = \inf_{j \geq 1} u_j$, where the functions $u_j: [0,\infty) \to \RR$ are defined in \eqref{UJ}. Each $u_j$ is positive, strictly increasing and convex on $[0,\infty)$, with $u_j(0)=c(\ve) j$ and $u'_j(0)=0$. Moreover, $u_j(r)$ is asymptotically linear as $r \to \infty$. Notice that for each $r \geq 0$,
\begin{equation}\label{tends}
\lim_{j \to \infty} u_j(r)=\infty, 
\end{equation}
hence the infimum in the definition of $\bb$ is in fact a minimum, that is, $\bb(r)=u_j(r)$ for some $j$ depending on $r$. Moreover, \eqref{tends} and the monotonicity of the $u_j$ easily imply that if $u_j(r) < u_k(r)$ for all $k \neq j$, then there is an open interval containing $r$ throughout which $u_j < u_k$ for all $k \neq j$. \vs

It will be convenient to say that $u_j$ is a {\bit constituent} of $\bb$ if $u_j=\bb$ in some non-empty open interval. In this case, we say that the index $j$ is {\bit present}. If $u_j$ is not a constituent of $\bb$, we say that $j$ is {\bit absent}. \vs

The following was first observed in \cite{Su}:

\begin{lemma}\label{paq}
If $j \geq 1$ is present, then $j=q_n$ for some $n \geq 0$.
\end{lemma}

\begin{proof}
Evidently $j=q_0=1$ is always present: Since $u_j(0)=c(\ve) j$, in some neighborhood of $0$ we have $u_1<u_j$ for all $j \geq 2$. Suppose $q_n<j<q_{n+1}$ for some $n \geq 0$. Then $\| q_n \alpha \| < \| j \alpha \| < 1/2$ since by \thmref{mom} the denominators $\{ q_n \}$ are the moments of closest return. It follows from monotonicity of $x \mapsto \sin^2 x$ on $[0,\pi/2]$ that
$$
\sin^2(\pi q_n \alpha) = \sin^2 (\pi \| q_n \alpha \|) < \sin^2 (\pi \| j \alpha \|) = \sin^2(\pi j \alpha).
$$
This easily implies $u_{q_n}<u_j$ everywhere, proving that $j$ is absent.
\end{proof}

Below we address the question of which elements of the sequence $\{ q_n \}$ are present. If $1 \leq q_n < q_m$, let $r_{n,m}$ denote the first coordinate of the unique intersection point of the graphs of $u_{q_n}$ and $u_{q_m}$. A simple computation based on the formula \eqref{UJ} shows that
\begin{equation}\label{rr1}
r_{n,m}^2 = \frac{q_m^2-q_n^2}{\delta_n - \delta_m},
\end{equation}
where
\begin{equation}\label{dn}
\delta_n = 4 \sin^2 (\pi q_n \alpha) = 4 \sin^2 (\pi \| q_n \alpha \|).
\end{equation}
Note that this defines $r_{n,m}$ for all pairs $0 \leq n < m$ except when $n=0, m=1$, and $1/2 < \alpha< 1$, in which case $q_0=q_1=1$. In this case, we set $r_{0,1}=0$. Evidently,
$$
\begin{cases} u_{q_n} < u_{q_m} & \quad \text{on} \ [0,r_{n,m}) \\
u_{q_n} > u_{q_m} & \quad \text{on} \ (r_{n,m},\infty).
\end{cases}
$$

In what follows, all triples $(k,n,m)$ of non-negative integers are assumed to be ordered in the sense that $k<n<m$.

\begin{lemma}[Trichotomy]\label{tric}
Suppose $q_k<q_n<q_m$. Then, the triple $(k,n,m)$ is of one of the following types: \vs
\begin{enumerate}
\item[$\bullet$]
{\bit Fair}, where $r_{k,n}<r_{k,m}<r_{n,m}$. In this case, $u_{q_n}< \min \{ u_{q_k}, u_{q_m} \}$ in the interval $(r_{k,n},r_{n,m})$. \vs
\item[$\bullet$]
{\bit Near miss}, where $r_{k,n}>r_{k,m}>r_{n,m}$. In this case, $u_{q_n} > \min \{ u_{q_k}, u_{q_m} \}$ everywhere. \vs
\item[$\bullet$]
{\bit Strike}, where $r_{k,n}=r_{k,m}=r_{n,m}$. In this case, $u_{q_n} > \min \{ u_{q_k}, u_{q_m} \}$ everywhere except at the common intersection point.\footnote{We don't know if strike triples can actually occur.}
\end{enumerate}
\end{lemma}

The proof is straightforward and will be left to the reader. The three possibilities are depicted in \figref{3curves}. \vs

Observe that \lemref{tric} covers all ordered triples $(k,n,m)$ of non-negative integers except when $k=0, n=1$, and $1/2<\alpha<1$, in which case $q_0=q_1=1$ and $r_{k,n}=0<r_{k,m}=r_{n,m}$. By convention, we consider $(0,1,m)$ a fair triple in this case.

\begin{figure*}
  \includegraphics[width=\textwidth]{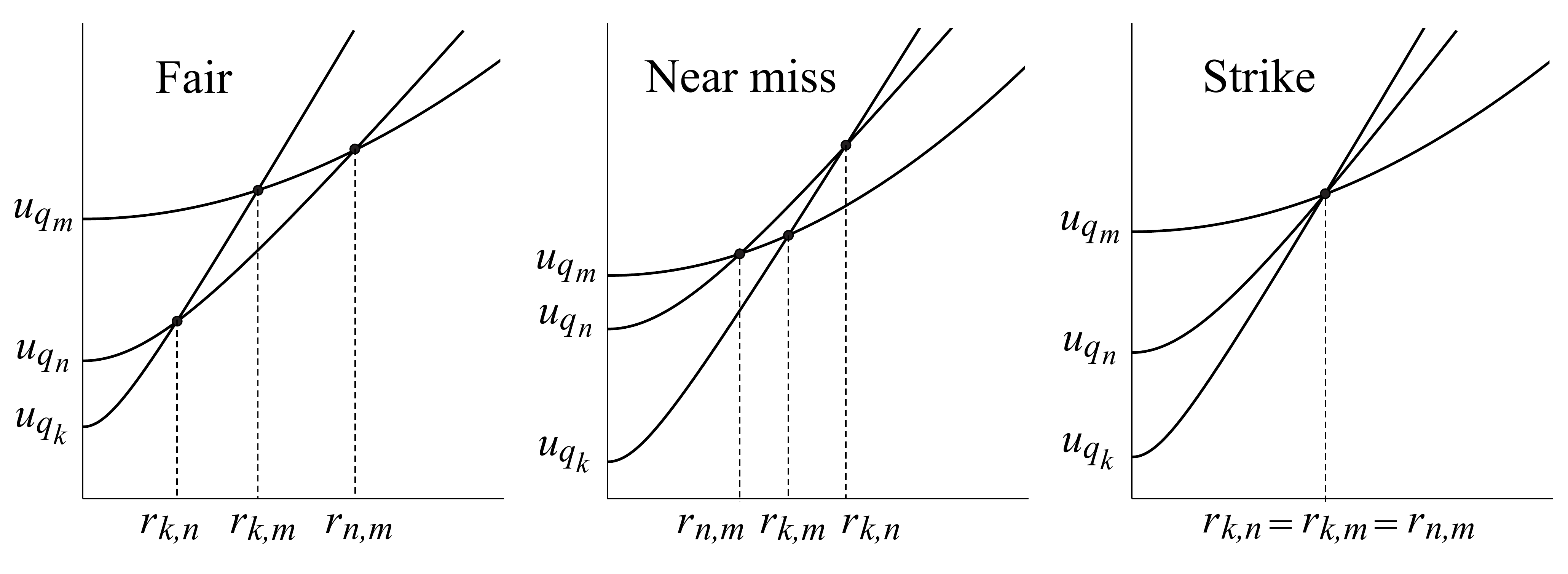}
\caption{Trichotomy of ordered triples $(k,n,m)$.}
\label{3curves} 
\end{figure*}

\begin{theorem}[Combinatorial characterization of presence]\label{chab}
For $n \geq 1$, the index $q_n$ is present if and only if every triple $(k,n,m)$ is fair.
\end{theorem}

\begin{proof}
If $1/2<\alpha<1$ so $q_0=q_1=1$, then $q_1$ is present and every triple $(0,1,m)$ is fair by our convention. Therefore, we may assume $n \geq 2$, or $n=1$ but $0<\alpha<1/2$. \vs

If there is a near miss or strike triple $(k,n,m)$, \lemref{tric} shows that $u_{q_n} > \min \{ u_{q_k}, u_{q_m} \}$ everywhere except possibly at one point. Thus, $q_n$ is absent. Conversely, suppose every triple $(k,n,m)$ is fair. Let $x=\max \, \{ r_{k,n} : k<n \}$ and $y=\min \, \{ r_{n,m}: m>n \}$ ($y$ exists since $r_{n,m} \to \infty$ as $m \to \infty$). Clearly, $x<y$. Moreover, $u_{q_n}<u_{q_k}$ in $(x,\infty)$ if $k<n$, and $u_{q_n}<u_{q_m}$ in $[0,y)$ if $m>n$. Thus, $u_{q_n} < u_{q_j}$ in $(x,y)$ if $j \neq n$, which proves $q_n$ is present.
\end{proof}

The following corollary is immediate from the above proof:

\begin{corollary}\label{who}
Suppose $q_n$ is present for some $n \geq 0$ and $[x,y]$ is the maximal interval on which $u_{q_n}=\bb$. Then
$$
x=\max \, \{ r_{k,n} : k<n \} \qquad \text{and} \qquad y = \min \, \{ r_{n,m}: m>n \}
$$
(for $n=0$, $x$ is understood to be $0$). Moreover, If $q_{n^*}$ is the next present index after $q_n$, then $r_{n,n^*} = y$.
\end{corollary}

\thmref{p/a} below will show that $n^*$ is always $n+1$ or $n+2$. \vs

Recall from \eqref{dn} that $\delta_n = 4 \sin^2 (\pi \| q_n \alpha \|)$. The sequence $\{ \delta_n \}_{n \geq 1}$ is strictly decreasing and tends to $0$ as $n \to \infty$.

\begin{lemma}\label{do}
For every $n \geq 1$, $\delta_n > 2 \delta_{n+2}$. If $a_{n+2} \geq 2$, the stronger inequality $\delta_n > 2 \delta_{n+1}$ holds.
The same is true for $n=0$ provided that $0<\alpha<1/2$.
\end{lemma}

\begin{proof}
We will use the elementary inequality
\begin{equation}\label{coss}
\sin^2 y > 2 \sin^2 x \qquad \Bigl( 0<2x<y<\frac{\pi}{2} \Bigr).
\end{equation}
If $n \geq 1$, or if $n=0$ and $0<\alpha<1/2$, then $\| q_n \alpha \| = a_{n+2} \| q_{n+1} \alpha \| + \| q_{n+2} \alpha \|$ (see \lemref{qnp}(iii) and the subsequent remark). This shows $\| q_n \alpha \| > 2 \| q_{n+2} \alpha \|$, and applying \eqref{coss} with $y=\pi \| q_n \alpha \|$ and $x=\pi \| q_{n+2} \alpha \|$ will give $\delta_n>2\delta_{n+2}$. If $a_{n+2} \geq 2$, we have $\| q_n \alpha \| > 2 \| q_{n+1} \alpha \|$, so \eqref{coss} with $y=\pi \| q_n \alpha \|$ and $x= \pi \| q_{n+1} \alpha \|$ will give $\delta_n>2\delta_{n+1}$.
\end{proof}

\begin{lemma}\label{se}
Let $n \geq 1$. 
\begin{enumerate}
\item[$\bullet$]
If $a_{n+2} \geq 2$, then $r_{n,n+1}<r_{n,m}$ for every $m>n+1$. \vs
\item[$\bullet$]
If $a_{n+2}=1$, then $r_{n,n+2}<r_{n,m}$ for every $m>n+2$.  . \vs
\end{enumerate}
The same is true for $n=0$ provided that $0 < \alpha <1/2$.
\end{lemma}

\begin{proof}
Let $n \geq 1$, or $n=0$ and $0<\alpha<1/2$. First suppose $a_{n+2} \geq 2$ and choose any $m>n+1$. By \lemref{do}, $\delta_n > 2 \delta_{n+1}$, so
\begin{equation}\label{dd1}
\delta_n-\delta_m < \delta_n < 2 (\delta_n-\delta_{n+1}).
\end{equation}
Also, $q_{n+2}=a_{n+2} q_{n+1}+q_n > 2 q_{n+1}$, so
\begin{equation}\label{qq1}
q_m^2 - q_n^2 \geq q_{n+2}^2 - q_n^2 > 4 q_{n+1}^2 - q_n^2 > 4 (q_{n+1}^2-q_n^2).
\end{equation}
It follows from \eqref{dd1}, \eqref{qq1}, and the formula \eqref{rr1} that
$$
r_{n,m}^2 = \frac{q_m^2-q_n^2}{\delta_n-\delta_m} > 2 \ \frac{q_{n+1}^2-q_n^2}{\delta_n-\delta_{n+1}} = 2 r_{n,n+1}^2,
$$
which proves $r_{n,n+1}<r_{n,m}$. \vs

Next suppose $a_{n+2}=1$ and choose any $m>n+2$. By \lemref{do}, $\delta_n > 2 \delta_{n+2}$, so
\begin{equation}\label{dd2}
\delta_n-\delta_m < \delta_n < 2 (\delta_n-\delta_{n+2}).
\end{equation}
Also, $q_{n+2}=q_{n+1}+q_n$, so
\begin{align}\label{qq2}
q_m^2 - q_n^2 \geq q_{n+3}^2 - q_n^2 & \geq (q_{n+2}+q_{n+1})^2-q_n^2 \notag \\
& = (2q_{n+1}+q_n)^2-q_n^2 = 4 q_{n+1}^2 + 4 q_{n+1}q_n \\
& > 2 (q_{n+1}^2+ 2 q_{n+1}q_n) = 2 (q_{n+2}^2-q_n^2). \notag
\end{align}
It follows from \eqref{dd2}, \eqref{qq2}, and the formula \eqref{rr1} that
$$
r_{n,m}^2 = \frac{q_m^2-q_n^2}{\delta_n-\delta_m} > \frac{q_{n+2}^2-q_n^2}{\delta_n-\delta_{n+2}} = r_{n,n+2}^2,
$$
which proves $r_{n,n+2}<r_{n,m}$.
\end{proof}

\begin{theorem}[No consecutive absentees]\label{p/a}
Suppose $q_n$ is present but $q_{n+1}$ is absent for some $n \geq 0$. Then $a_{n+2}=1$ and $q_{n+2}$ is present.
\end{theorem}

\begin{proof}
If $n=0$ and $1/2<\alpha<1$, then $a_1=1$ and $q_1=q_0=1$ is present. Therefore, we may assume that $n \geq 1$, or $n=0$ and $0<\alpha<1/2$. \vs

If $a_{n+2}$ were $\geq 2$, \lemref{se} would guarantee that $r_{n,n+1}<r_{n,m}$ for every $m>n+1$. This, by \corref{who} would imply that $q_{n+1}$ is the next present index after $q_n$, contradicting our assumption. Thus, $a_{n+2}=1$. Invoking \lemref{se} once more, we see that $r_{n,n+2}<r_{n,m}$ for every $m>n+2$. Since $q_{n+1}$ is absent, \corref{who} shows that $q_{n+2}$ is the next present index after $q_n$.
\end{proof}

We can now reduce the characterization of presence in \thmref{chab} to a much simpler condition:

\begin{corollary}\label{sufi}
For $n \geq 1$, the index $q_n$ is present if and only if $a_{n+1} \geq 2$ or the triple $(n-1,n,n+1)$ is fair.
\end{corollary}

\begin{proof}
If $q_n$ is present, the triple $(n-1,n,n+1)$ is fair by \thmref{chab}.
If $q_n$ is absent, by \thmref{p/a} both $q_{n-1}$ and $q_{n+1}$ must be present and $a_{n+1}=1$. Moreover, $r_{n-1,n+1} \leq r_{n-1,n}$ by \corref{who}, which shows the triple $(n-1,n,n+1)$ is a near miss or strike.
\end{proof}

A purely arithmetical characterization of presence is far from simple; see the Appendix for further details on this problem.

\section{Asymptotic analysis of the boundary function}
\label{sec:aa}

We continue assuming that $\alpha \in \TT$ is a fixed irrational number with the partial quotients $\{ a_n \}$ and rational convergents $\{ p_n/q_n \}$. Suppose $u_{q_n}$ is a constituent of the boundary function $\bb$ and $[x,y]$ is the maximal interval on which $u_{q_n}=\bb$. \corref{who} combined with \thmref{p/a} show that $y=r_{n,n+1}$ or $r_{n,n+2}$, and $x=r_{n-1,n}$ or $r_{n-2,n}$. Our next task is to find sharp asymptotics for these endpoints. \vs

We will make use of the following convenient terminologies and notations. By a {\bit universal constant} we mean one which is independent of all the parameters and variables involved. Given positive sequences $\{ a_n \}$ and $\{ b_n \}$, we write $a_n \preccurlyeq b_n$ if there is a universal constant $C>0$ such that $a_n \leq C b_n$ for all large $n$. This may also be written as $b_n \succcurlyeq a_n$. The notation $a_n \asymp b_n$ means that both $a_n \preccurlyeq b_n$ and $a_n \succcurlyeq b_n$ hold. In other words, $a_n \asymp b_n$ if there is a universal constant $C \geq 1$ such that $C^{-1} \, b_n \leq a_n \leq C \, b_n$ for all large $n$. For a pair of positive functions $f(r),g(r)$ depending on $r>0$, we define $f(r) \preccurlyeq g(r)$ and $f(r) \asymp g(r)$ similarly, where now the corresponding inequalities should hold for all large $r$.
Any such relation will be called an {\bit asymptotically universal bound}.

\begin{lemma}\label{tn1}
As $n \to \infty$, the following asymptotically universal bound holds:
$$
r_{n,n+2} \asymp q_{n+1} q_{n+2}.
$$
\end{lemma}

\begin{proof}
The inequality
$$
q_{n+2} \geq q_{n+1}+q_n > 2q_n,
$$
shows that
$$
\frac{3}{4} q_{n+2}^2 < q_{n+2}^2-q_n^2 < q_{n+2}^2.
$$
Also, $\delta_n > 2 \delta_{n+2}$ by \lemref{do}, so
$$
\frac{1}{2} \delta_n < \delta_n - \delta_{n+2} < \delta_n.
$$
Finally, the expansion $\sin^2 x = x^2 + o(x^2)$ as $x \to 0$ together with \lemref{qnp}(ii) show that
$$
\delta_n \asymp \| q_n \alpha \|^2 \asymp \frac{1}{q_{n+1}^2}.
$$
Putting these results together, and using the formula \eqref{rr1}, we obtain
\[
r_{n,n+2}^2 = \frac{q_{n+2}^2-q_n^2}{\delta_n - \delta_{n+2}} \asymp \frac{q_{n+2}^2}{\delta_n} \asymp q_{n+1}^2 q_{n+2}^2. \hfill \qedhere
\]
\end{proof}

Finding sharp asymptotics for $r_{n,n+1}$ is slightly more subtle. We first need a preliminary estimate:

\begin{lemma}\label{del}
As $n \to \infty$, the following asymptotically universal bound holds:
$$
\delta_n - \delta_{n+1} \asymp
\begin{cases}
\dfrac{1}{q_{n+1}^2} & \qquad \text{if} \ a_{n+2} \geq 2 \vs \\
\dfrac{1}{q_{n+1} q_{n+3}} & \qquad \text{if} \ a_{n+2} =1.
\end{cases}
$$
\end{lemma}

\begin{proof}
If $a_{n+2} \geq 2$, \lemref{do} shows that $\delta_n > 2\delta_{n+1}$. Hence,
$$
\delta_n -\delta_{n+1} \asymp \delta_n \asymp \frac{1}{q_{n+1}^2}.
$$
Next, suppose $a_{n+2}=1$. By the mean value theorem, if $0<x<y<\pi/4$,
$$
\frac{\sin^2 y - \sin^2 x}{y-x} = \sin (2z) \qquad \text{for some} \ z \in (x,y).
$$
Since $4z/\pi < \sin (2z) < 2z$ for $0<z<\pi/4$, this gives
$$
\frac{4}{\pi} x < \frac{\sin^2 y - \sin^2 x}{y-x} < 2y \qquad \Bigl( 0<x<y<\frac{\pi}{4} \Bigr).
$$
Substituting $x=\pi \| q_{n+1} \alpha \|$ and $y=\pi \| q_n \alpha \|$, which are both in $(0,\pi/4)$ for $n \geq 3$, we obtain
$$
16 \pi \| q_{n+1} \alpha \| < \frac{\delta_n-\delta_{n+1}}{\| q_n \alpha \| - \| q_{n+1} \alpha \|} < 8\pi^2 \| q_n \alpha \|.
$$
But by \lemref{qnp}(iii), $a_{n+2}=1$ implies $\| q_n \alpha \| = \| q_{n+1} \alpha \| + \| q_{n+2} \alpha \| < 2 \| q_{n+1} \alpha \|$, which shows
$$
8\pi \| q_n \alpha \| \ \| q_{n+2} \alpha \| < \delta_n - \delta_{n+1} < 8\pi^2 \| q_n \alpha \| \ \| q_{n+2} \alpha \|.
$$
Thus, by \lemref{qnp}(ii),
\[
\delta_n - \delta_{n+1} \asymp \| q_n \alpha \| \ \| q_{n+2} \alpha \| \asymp \frac{1}{q_{n+1} q_{n+3}}. \hfill \qedhere
\]
\end{proof}

\begin{lemma}\label{ttn}
As $n \to \infty$, the following asymptotically universal bound holds:
$$
r_{n,n+1} \asymp
\begin{cases}
q_{n+1}^2 & \qquad \text{if} \ a_{n+2} \geq 2, a_{n+1} \geq 2 \vs \\
q_{n+1}^{3/2} \, q_{n-1}^{1/2} & \qquad \text{if} \ a_{n+2} \geq 2, a_{n+1} =1  \vs \\
q_{n+1}^{3/2} \, q_{n+3}^{1/2} & \qquad \text{if} \ a_{n+2} =1, a_{n+1} \geq 2  \vs \\
q_{n+1} \, q_{n-1}^{1/2} \, q_{n+3}^{1/2} & \qquad \text{if} \ a_{n+2} =1, a_{n+1} =1 \\
\end{cases}
$$
\end{lemma}

\begin{proof}
Since $q_{n+1}=a_{n+1}q_n+q_{n-1}$, it is easy to see that
$$
q_{n+1}^2 - q_n^2 \asymp
\begin{cases}
q_{n+1}^2 & \qquad \text{if} \ a_{n+1} \geq 2 \vs \\
q_{n+1} q_{n-1} & \qquad \text{if} \ a_{n+1} =1.
\end{cases}
$$
The result follows from this, \lemref{del}, and the formula \eqref{rr1} for $r_{n,n+1}$.
\end{proof}

\begin{theorem}\label{fund}
Suppose $u_{q_n}$ is a constituent of $\bb$ and $[x,y]$ is the maximal interval on which $u_{q_n}=\bb$. Then, as $n \to \infty$, the following asymptotically universal bounds hold:
$$
x \asymp q_n^2 \qquad \text{and} \qquad y \asymp q_{n+1}^2.
$$
\end{theorem}

\begin{proof}
By \corref{who} and \thmref{p/a}, $y=r_{n,n+1}$ if $q_{n+1}$ is present, and $y=r_{n,n+2}$ if $q_{n+1}$ is absent. Similarly, $x=r_{n-1,n}$ if $q_{n-1}$ is present, and $x=r_{n-2,n}$ if $q_{n-1}$ is absent. Therefore, to prove the theorem, we need to show that as $n \to \infty$,
$$
\begin{cases}
r_{n,n+1} \asymp q_{n+1}^2 & \quad \text{if} \ q_{n+1} \ \text{is present} \\
r_{n,n+2} \asymp q_{n+1}^2 & \quad \text{if} \ q_{n+1} \ \text{is absent}
\end{cases}
$$
The second case is almost immediate: If $q_{n+1}$ is absent, then $a_{n+2}=1$ by \thmref{p/a}, which gives $q_{n+1}<q_{n+2}=q_{n+1}+q_n<2q_{n+1}$, or $q_{n+2} \asymp q_{n+1}$. Hence, by \lemref{tn1},
$$
r_{n,n+2} \asymp q_{n+1} q_{n+2} \asymp q_{n+1}^2.
$$
Let us then consider the first case, where $q_{n+1}$ is present. Of the four estimates for $r_{n,n+1}$ covered by \lemref{ttn}, the first is automatic, so let us consider the remaining three cases: \vs

$\bullet$ Case A: $a_{n+2} \geq 2, a_{n+1} =1$. Since $q_n$ is present, we must have $a_n \leq 4$ (see the Appendix). Hence
$$
q_{n+1}=q_n+q_{n-1} = (a_n q_{n-1}+q_{n-2})+q_{n-1} < 6 q_{n-1},
$$
which shows
$$
r_{n,n+1} \asymp q_{n+1}^{3/2} \, q_{n-1}^{1/2} \asymp q_{n+1}^2. \vs
$$

$\bullet$ Case B: $a_{n+2} = 1, a_{n+1} \geq 2$. Since $q_{n+1}$ is present, we must have $a_{n+3} \leq 4$ (again, see the Appendix). Hence
$$
q_{n+3}=a_{n+3} q_{n+2}+q_{n+1} \leq 4q_{n+2}+q_{n+1} = 4(q_{n+1}+q_n)+q_{n+1} < 6 q_{n+1},
$$
which shows
$$
r_{n,n+1} \asymp q_{n+1}^{3/2} \, q_{n+3}^{1/2} \asymp q_{n+1}^2. \vs
$$

$\bullet$ Case C: $a_{n+2} = a_{n+1} = 1$. We consider two sub-cases: If $a_n \geq 2$, then $q_{n-1}$ is present by \corref{sufi}. By case B (applied to $n-1$), $r_{n-1,n} \asymp q_n^2$. Since $r_{n-1,n} \leq r_{n,n+1} \leq r_{n,n+2}$, we have
$$
q_n^2 \preccurlyeq r_{n,n+1} \preccurlyeq q_{n+1}q_{n+2}.
$$
But $a_{n+1}=a_{n+2}=1$ shows that $q_n \asymp q_{n+1} \asymp q_{n+2}$, which proves $r_{n,n+1} \asymp q_{n+1}^2$. \vs

On the other hand, if $a_n=1$, we use $r_{n-2,n} \leq r_{n,n+1} \leq r_{n,n+2}$, which by \lemref{tn1} shows
$$
q_{n-1} q_n \preccurlyeq r_{n,n+1} \preccurlyeq q_{n+1}q_{n+2}.
$$
But $a_n=a_{n+1}=a_{n+2}=1$ shows that $q_{n-1} \asymp q_n \asymp q_{n+1} \asymp q_{n+2}$, which proves $r_{n,n+1} \asymp q_{n+1}^2$.
\end{proof}

We now have all the necessary ingredients to prove Theorem A in \S \ref{sec:intro}. Recall from \S \ref{sec:p} that ${\mathscr D}_\nu$ is the set of Diophantine numbers of exponent $\nu \geq 2$.

\begin{proof}[Proof of Theorem A]
Set $\bb=\bb_\alpha$. Given large $r>0$, suppose $\bb(r)=u_{q_n}(r)$ for some $n$. Let $[x,y]$ be the maximal interval on which $\bb=u_{q_n}$. By \thmref{fund}, $x \asymp q_n^2$ and $y \asymp q_{n+1}^2$. Hence,
$$
\bb(x) = c(\ve) \sqrt{\delta_n x^2 + q_n^2} \asymp \sqrt{\frac{q_n^4}{q_{n+1}^2}+q_n^2} \asymp q_n \asymp \sqrt{x}
$$
and
$$
\bb(y)= c(\ve) \sqrt{\delta_n y^2 + q_n^2} \asymp \sqrt{\frac{q_{n+1}^4}{q_{n+1}^2}+q_n^2} \asymp q_{n+1} \asymp \sqrt{y}.
$$
Thus, at the end points of the interval $[x,y]$, $\bb$ is comparable to the square root function. Since $\bb=u_{q_n}$ is convex and the square root function is concave on this interval, it follows that $\bb(r) \preccurlyeq \sqrt{r}$ for all $r \in [x,y]$. \vs

Now suppose $\alpha \in {\mathscr D}_\nu$, so $q_{n+1} \preccurlyeq q_n^{\nu-1}$. Then, with $u_{q_n}$ and $[x,y]$ as above and $r \in [x,y]$, we have
$$
\bb(r) \succcurlyeq \bb(x) \asymp \sqrt{x} \asymp q_n \succcurlyeq q_{n+1}^{1/(\nu-1)} \asymp y^{1/(2\nu-2)} \succcurlyeq r^{1/(2\nu-2)},
$$
as required.
\end{proof}

\begin{corollary}\label{log}
For Lebesgue almost every irrational number $\alpha \in \TT$,
$$
\lim_{r \to \infty} \frac{\log \bb_\alpha(r)}{\log r} = \frac{1}{2}.
$$
\end{corollary}

\begin{proof}
Let $\alpha$ belong to the full-measure set $\bigcap_{\nu>2} {\mathscr D}_{\nu}$. By Theorem A, for every $\nu>2$ there are positive constants $K_1=K_1(\nu)$, $K_2=K_2(\nu)$, and $R=R(\nu)$ such that $\bb=\bb_\alpha$ satisfies
$$
K_1 + \frac{1}{2\nu-2} \log r \leq \log \bb(r) \leq K_2 + \frac{1}{2} \log r \qquad \text{for all} \ r>R.
$$
This gives
$$
\frac{1}{2\nu-2} \leq \liminf_{r \to \infty} \frac{\log \bb(r)}{\log r} \leq \limsup_{r \to \infty} \frac{\log \bb(r)}{\log r} \leq \frac{1}{2}
$$
and the result follows by letting $\nu \to 2^+$.
\end{proof}

By contrast, we can construct irrationals of Liouville type for which the boundary function has arbitrarily slow growth over long intervals, and the construction is quite flexible. As an example, we prove the following

\begin{theorem}\label{lio}
There exist irrational numbers $\alpha \in \TT$ for which
$$
\liminf_{r \to \infty} \frac{\log \bb_\alpha(r)}{\log r} = 0 < \frac{1}{2} = \limsup_{r \to \infty} \frac{\log \bb_\alpha(r)}{\log r}.
$$
\end{theorem}

\begin{proof}
Set $\bb=\bb_\alpha$. By Theorem A, $\limsup_{r \to \infty} \log \bb(r)/\log r \leq 1/2$ for every irrational $\alpha$. Furthermore, if $u_{q_n}=\bb$ on the maximal interval $[x,y]$, the proof of Theorem A shows that $\bb(x) \asymp \sqrt{x}$ and $\bb(y) \asymp \sqrt{y}$. It follows that
$$
\limsup_{r \to \infty} \frac{\log \bb(r)}{\log r}=\frac{1}{2}.
$$
Now suppose $\alpha$ is an irrational whose partial quotients $a_n$ grow so fast that $q_{n+1}$ is of the order of $\exp(q_n)$. Let $u_{q_n}=\bb$ on the maximal interval $[x,y]$ and $z=\sqrt{xy}$. By \thmref{fund}, $x \asymp q_n^2$ and $y \asymp q_{n+1}^2$, so $z \asymp q_n q_{n+1}$. This gives the asymptotic bound
$$
\bb(z) = c(\ve) \sqrt{\delta_n z^2 + q_n^2} \asymp \sqrt{\frac{q_n^2 q_{n+1}^2}{q_{n+1}^2}+q_n^2} \asymp q_n.
$$
Since $\log z \asymp \log q_n + \log q_{n+1} \asymp \log q_n + q_n \asymp q_n$, it follows that
$\log \bb(z) / \log z \asymp (\log q_n)/q_n$. Thus, $\liminf_{r \to \infty} \log \bb(r)/\log r = 0$.
\end{proof}

\begin{figure*}
\centering{\includegraphics[width=0.7\textwidth]{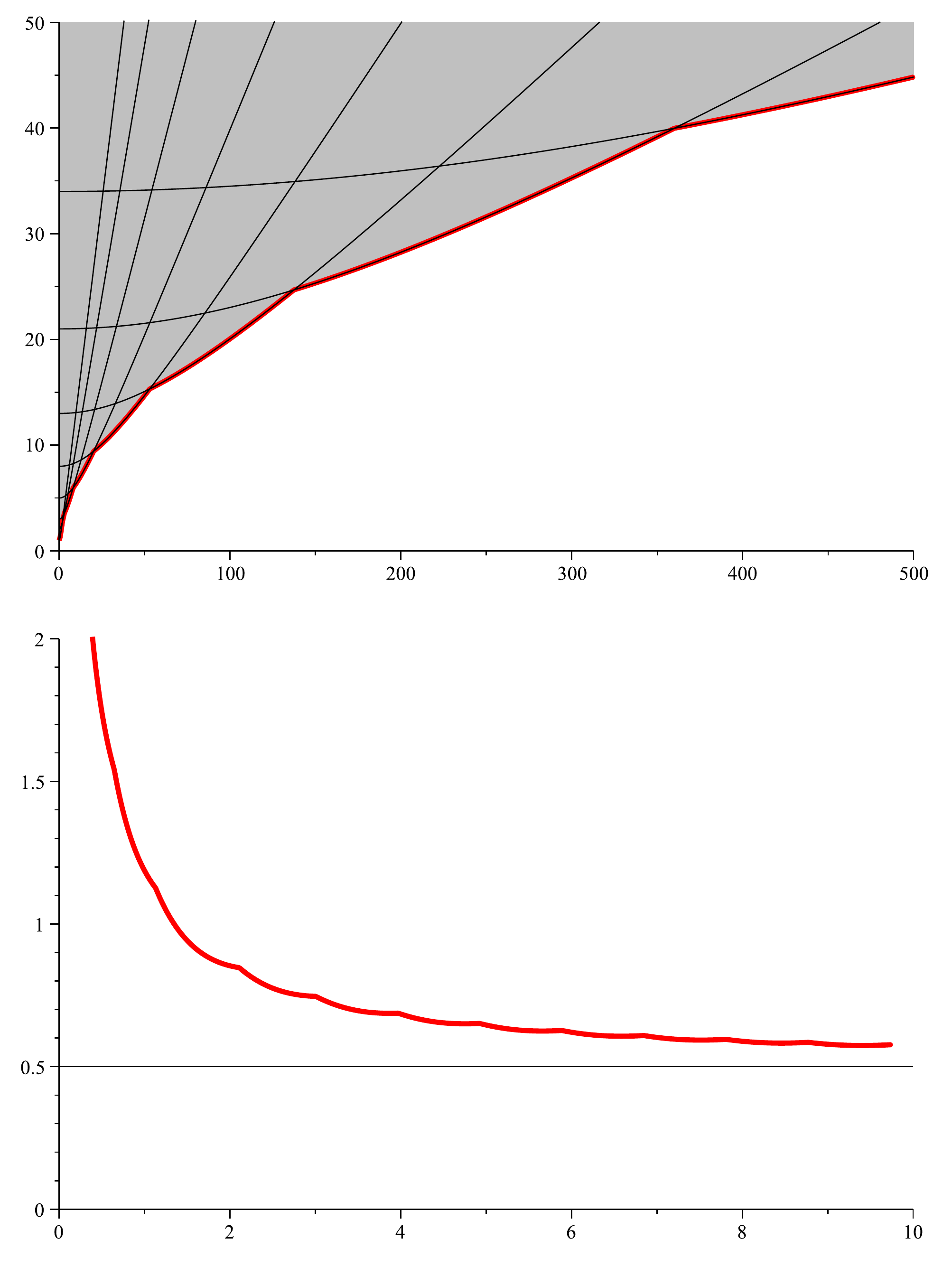}}
\caption{Top: Graphs of the boundary function $\bb_\alpha$ and $\{ u_{\alpha,q_n} \}_{1 \leq n \leq 8}$ for the golden mean $\alpha=[1,1,1,\ldots]$ (we have normalized so $c(\ve)=1$). Here every denominator $q_n$ is present. Bottom: Graph of $s \mapsto \log(\bb_\alpha(e^s))/s$.}
\label{GO} 
\end{figure*}

\begin{figure*}
\centering{\includegraphics[width=0.7\textwidth]{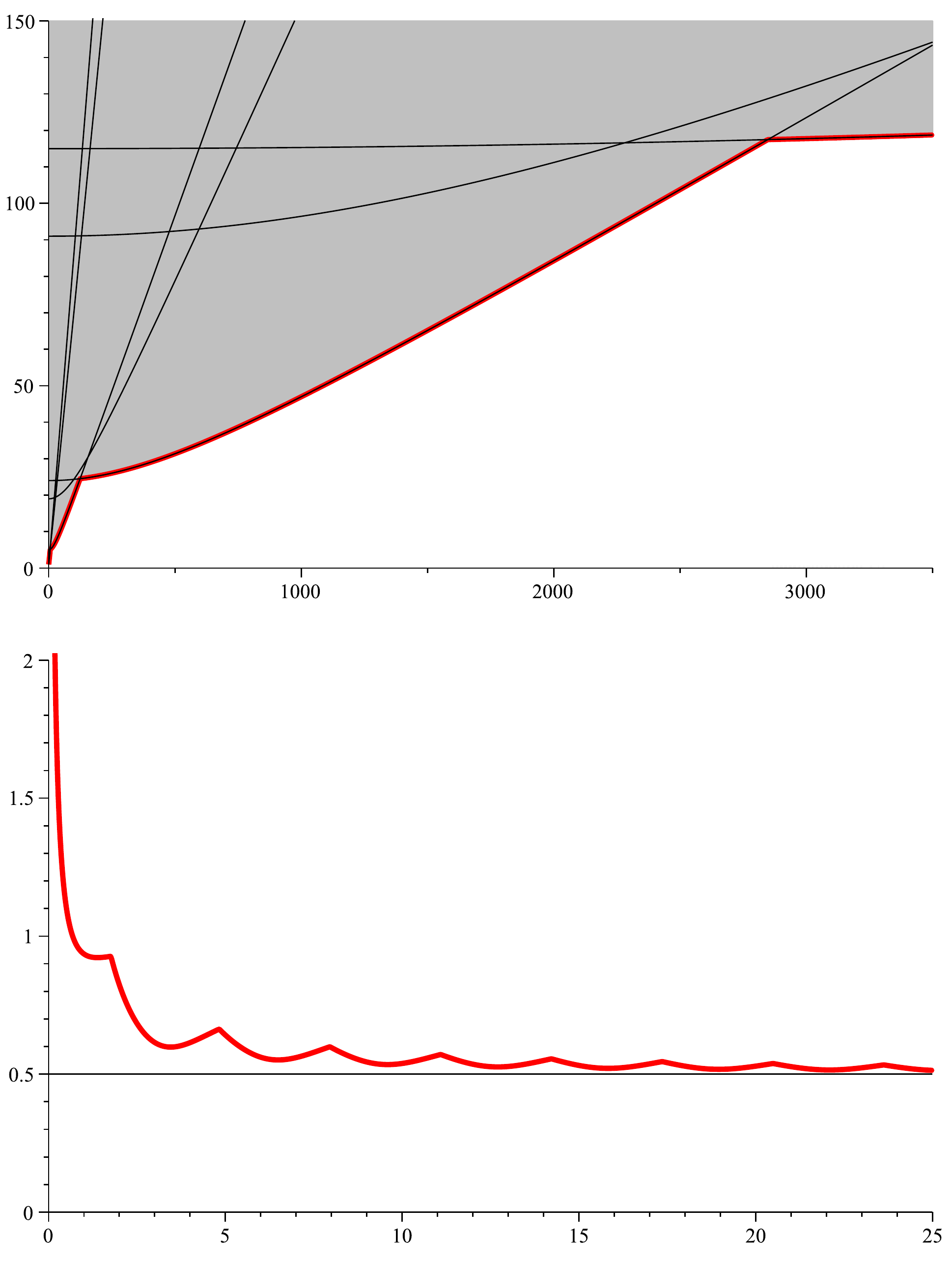}}
\caption{Top: Graphs of the boundary function $\bb_\alpha$ and $\{ u_{\alpha,q_n} \}_{1 \leq n \leq 7}$ for $\alpha=[1,3,1,3,\ldots]$ (we have normalized so $c(\ve)=1$). Here only the $q_n$ with odd $n$ are present. Bottom: Graph of $s \mapsto \log(\bb_\alpha(e^s))/s$.}
\label{TO} 
\end{figure*}

The following theorem illustrates the special role played by bounded type irrationals in this context. It is a sharpened version of the main result of \cite{E}.

\begin{theorem}[Optimality of bounded type]\label{D2}
The asymptotic bound
$$
\bb_\alpha(r) \asymp \sqrt{r} \qquad \text{as} \ r \to \infty
$$
holds if and only if $\alpha \in {\mathscr D}_2$.
\end{theorem}

\begin{proof}
The ``if'' part follows from Theorem A. For the ``only if'' part, we need to show that the sequence $\{ a_n \}$ of partial quotients of $\alpha$ is bounded if $\bb=\bb_\alpha$ satisfies $\bb(r) \asymp \sqrt{r}$. Take a large $n$. If $a_{n+1}=1$ there is nothing to prove. If $a_{n+1} \geq 2$, \corref{sufi} shows that $q_n$ is present. As before, let $[x,y]$ be the maximal interval on which $u_{q_n}=\bb$, and set $z=\sqrt{xy}$. As in the proof of \thmref{lio}, we have $z \asymp q_n q_{n+1}$ and $\bb(z) \asymp q_n$. The assumption $\bb(z) \asymp \sqrt{z}$ then implies $q_n q_{n+1} \asymp q_n^2$ or $q_{n+1} \asymp q_n$, proving that $a_{n+1}$ is bounded.
\end{proof}

Two examples illustrating the above results are shown in \figref{GO} and \figref{TO}. For the golden mean angle $\alpha=(\sqrt{5}-1)/2=[1,1,1,\ldots]$, all the $q_n$ are present, whereas for $\alpha=(\sqrt{21}-3)/2=[1,3,1,3,\ldots]$, only the $q_n$ with odd $n$ are present (these assertions will be justified in the Appendix). In both cases, the boundary function $\bb_\alpha(r)$ is asymptotic to $\sqrt{r}$ by \thmref{D2}.

\subsection*{Numerical comments}

The Margulis constant in dimension $4$ is known to satisfy
$$
\ve_4 \geq \frac{\sqrt{3}}{9\pi} = 0.061258\cdots
$$
(see \cite{Ke}) so the constant $c(\ve)$ in the definition of the boundary function can be taken
\begin{equation}\label{CE}
c(\ve) = \frac{1}{\sqrt{2 \cosh(\sqrt{3}/(9\pi))-2}} = 16.321642\cdots
\end{equation}
A tedious but completely straightforward re-working of the proofs of \lemref{tn1}, \lemref{del}, \lemref{ttn}, and \thmref{fund}, the details of which we omit here, yields the following explicit constants in the corresponding inequalities: \vs \\
$\bullet$ \lemref{tn1}:
$$
\frac{\sqrt{3}}{4\pi} \ q_{n+1} q_{n+2} \leq r_{n,n+2} \leq \frac{1}{\sqrt{2}}\ q_{n+1} q_{n+2}. \vs
$$
$\bullet$ \lemref{del}:
$$
\begin{cases}
\dfrac{2}{q_{n+1}^2} \leq \delta_n - \delta_{n+1} \leq \dfrac{4 \pi^2}{q_{n+1}^2} & \quad \text{if} \ a_{n+2} \geq 2 \vs \\
\dfrac{2\pi}{q_{n+1} q_{n+3}} \leq \delta_n - \delta_{n+1} \leq \dfrac{8\pi^2}{q_{n+1} q_{n+3}} & \quad \text{if} \ a_{n+2} =1, n \geq 3.
\end{cases}
$$
$\bullet$ \lemref{ttn}:
$$
\begin{cases}
\dfrac{\sqrt{3}}{4\pi} \ q_{n+1}^2 \leq r_{n,n+1} \leq \dfrac{1}{\sqrt{2}} \  q_{n+1}^2 & \quad \text{if} \ a_{n+2} \geq 2, a_{n+1} \geq 2 \vs \\
\dfrac{1}{2\pi} \ q_{n+1}^{3/2} \, q_{n-1}^{1/2} \leq r_{n,n+1} \leq  q_{n+1}^{3/2} \, q_{n-1}^{1/2} & \quad \text{if} \ a_{n+2} \geq 2, a_{n+1} =1  \vs \\
\dfrac{\sqrt{3}}{\sqrt{32}\pi} \ q_{n+1}^{3/2} \, q_{n+3}^{1/2} \leq r_{n,n+1} \leq \dfrac{1}{\sqrt{2\pi}} \ q_{n+1}^{3/2} \, q_{n+3}^{1/2} & \quad \text{if} \ a_{n+2} =1, a_{n+1} \geq 2, n \geq 3  \vs \\
\dfrac{1}{\sqrt{8}\pi} \ q_{n+1} \, q_{n-1}^{1/2} \, q_{n+3}^{1/2} \leq r_{n,n+1} \leq \dfrac{1}{\sqrt{\pi}} \ q_{n+1} \, q_{n-1}^{1/2} \, q_{n+3}^{1/2} & \quad \text{if} \ a_{n+2} =1, a_{n+1} =1, n \geq 3 \\
\end{cases}
$$
$\bullet$ \thmref{fund}:
If $q_{n+1}$ is absent,
$$
\frac{\sqrt{3}}{4\pi} \leq \frac{r_{n,n+2}}{q_{n+1}^2} \leq \sqrt{2},
$$
and if $q_{n+1}$ is present,
$$
\begin{cases}
\dfrac{\sqrt{3}}{4\pi} \leq \dfrac{r_{n,n+1}}{q_{n+1}^2} \leq \dfrac{1}{\sqrt{2}}  & \quad \text{if} \ a_{n+2} \geq 2, a_{n+1} \geq 2 \vs \\
\dfrac{1}{\sqrt{24}\pi} \leq \dfrac{r_{n,n+1}}{q_{n+1}^2} \leq  \dfrac{1}{\sqrt{2}} & \quad \text{if} \ a_{n+2} \geq 2, a_{n+1} =1, n \geq 5  \vs \\
\dfrac{\sqrt{3}}{4\pi} \leq \dfrac{r_{n,n+1}}{q_{n+1}^2} \leq \dfrac{\sqrt{3}}{\sqrt{\pi}} & \quad \text{if} \ a_{n+2} =1, a_{n+1} \geq 2, n \geq 4  \vs \\
\dfrac{\sqrt{3}}{24\pi} \leq \dfrac{r_{n,n+1}}{q_{n+1}^2} \leq \sqrt{2} & \quad \text{if} \ a_{n+2} =1, a_{n+1} =1, n \geq 5 \\
\end{cases}
$$
Taking the worst case scenario, it follows that the endpoints $x,y$ in \thmref{fund} satisfy
\begin{align*}
\dfrac{\sqrt{3}}{24\pi} \ q_n^2 & \leq x \leq \sqrt{2} \ q_n^2 & \text{if} \ n \geq 6\\
\dfrac{\sqrt{3}}{24\pi} \ q_{n+1}^2 & \leq y \leq \sqrt{2} \ q_{n+1}^2 & \text{if} \ n \geq 5
\end{align*}
These estimates, together with \eqref{CE}, show that
\begin{align*}
\bb(x) = c(\ve) \sqrt{\delta_n x^2 + q_n^2} & \leq c(\ve) \sqrt{\frac{8 \pi^2 q_n^4}{q_{n+1}^2}+q_n^2} \leq c(\ve) \sqrt{8\pi^2+1} \,  q_n \\
& \leq c(\ve) \sqrt{8\pi^2+1} \, \frac{\sqrt{24\pi}}{\sqrt[4]{3}} \, \sqrt{x} \leq 1000 \sqrt{x}.
\end{align*}
and
\begin{align*}
\bb(y) = c(\ve) \sqrt{\delta_n y^2 + q_n^2} & \leq c(\ve) \sqrt{8 \pi^2 q_{n+1}^2+q_n^2} \leq c(\ve) \sqrt{8\pi^2+1} \,  q_{n+1} \\
& \leq c(\ve) \sqrt{8\pi^2+1} \, \frac{\sqrt{24\pi}}{\sqrt[4]{3}} \, \sqrt{y} \leq 1000 \sqrt{y}.
\end{align*}
It follows that $\bb(r) \leq 1000 \sqrt{r}$ if $\bb(r)=u_{q_n}(r)$ for some $n \geq 6$. Since either $q_6$ or $q_7$ is present, the latter condition must hold as soon as $r \geq \sqrt{2} \, q_7^2$.

\begin{corollary}[Universal upper bound]\label{expl}
For every irrational $\alpha \in \TT$, the boundary function $\bb_\alpha$ satisfies
$$
\bb_\alpha(r) \leq 1000 \sqrt{r} \qquad \text{if} \quad r \geq \sqrt{2} \, q_7^2.
$$
\end{corollary}

\section{Geometry of Margulis cusps}
\label{sec:geo}

Let $\Gamma \subset \isom$ be a discrete group acting freely on $\HH^4$, and $\infty \in \bd \HH^4$ be a parabolic fixed point of $\Gamma$. By \thmref{cyclic} the stabilizer subgroup $\Gamma_\infty \subset \Gamma$ is cyclic when it contains an irrational screw translation, so after a suitable change of coordinates we can assume that $\Gamma_\infty$ is generated by the map $g_{\alpha}$ of \eqref{ga} for some irrational number $\alpha \in \TT$. Recall that the Margulis region associated with the parabolic fixed point $\infty$ is given by
\begin{equation}\label{MA}
T_\alpha = \{ (r,\theta,z,t) \in \HH^4:  t > \bb_\alpha(r) \}.
\end{equation}
where $\bb_\alpha :[0, \infty) \to \RR$ is defined by \eqref{psi}. The Margulis cusp $C_{\alpha} = T_\alpha / \langle g_\alpha \rangle$ embeds isometrically into the hyperbolic manifold $M=\HH^4/\Gamma$ and forms a connected component of its thin part $\mt$ (see \S \ref{sec:p}). Note that $C_\alpha$ is a universal model depending only on the rotation angle $\alpha$, and in particular it is independent of the rest of the group $\Gamma$ and the manifold $M$. \vs

We remark that the cusp $C_{\alpha}$ is always a {\bit uniformly quasiconvex} subset of $M$ in the following sense: \vs

\noindent
{\it There is a universal $\delta>0$ with the property that any pair of points in $C_{\alpha}$ can be joined by a geodesic in $M$ which stays within the distance $\delta$ from $C_{\alpha}$.} \vs

\noindent
To see this, lift the given points in $C_{\alpha}$ to a pair of points $p,q \in T_{\alpha}$. In the geodesic triangle formed by $p,q,\infty$, the two vertical sides $[p,\infty)$ and $[q,\infty)$ are contained in $T_{\alpha}$ because of \eqref{MA}. There is a universal $\delta>0$ for which all geodesic triangles in $\HH^4$ are $\delta$-thin in the sense of Gromov. For this $\delta$, every point of the third side $[p,q]$ lies within the distance $\delta$ from the union of the other two sides, hence from $T_{\alpha}$. The image of $[p,q]$ in $M$ is then the desired geodesic.       

\subsection*{Topology of Margulis cusps}

For every irrational $\alpha \in \TT$, the Margulis cusp $C_{\alpha}$ is homeomorphic to the product $D^2 \times S^1 \times \RR \cong \RR^3 \times S^1$ and thus has the homotopy type of the circle. For each $t>c(\ve)$ the horosphere at height $t$ based at $\infty$ intersects $T_\alpha$ along the solid cylinder
$$
\hat{L}_t = \{ (r,\theta,z,t) : r < \bb_{\alpha}^{-1}(t) \} \cong D^2 \times \RR,
$$
so the Margulis region $T_\alpha= \bigcup_{t>c(\ve)} \hat{L}_t$ is homeomorphic to $D^2 \times \RR \times \RR$. The horizontal foliation of $T_\alpha$ by the $\hat{L}_t$ is leafwise invariant under the action of $g_\alpha$, so it descends to a product foliation $\ff$ of $C_{\alpha}$ whose leaves are $3$-dimensional solid tori $L_t = \hat{L}_t/ \langle g_\alpha \rangle \cong D^2 \times S^1$ (see \figref{HOR}). The volume of the leaf $L_t$ grows at least linearly in $t$. To see this, note that $L_t$ can be identified with the solid cylinder $\{ (r,\theta,z,t) : r < \bb_{\alpha}^{-1}(t), 0 \leq z \leq 1 \}$ where the two ends are glued by $(r,\theta,0,t) \sim (r, \theta+\alpha, 1,t)$. Hence,
\begin{equation}\label{vol}
\text{vol}(L_t) =  \frac{\pi (\bb_{\alpha}^{-1}(t))^2}{t^3}.
\end{equation}
Since $\bb_{\alpha}(r) \preccurlyeq \sqrt{r}$ by Theorem A, we have $\bb_{\alpha}^{-1}(t) \succcurlyeq t^2$, which shows $\text{vol}(L_t) \succcurlyeq t$. By contrast, the core curve of $L_t$ shrinks as $t \to \infty$ since it can be identified with the segment $\{ (0,\theta,z,t): 0 \leq z \leq 1 \}$ with the two ends glued, and therefore has hyperbolic length $1/t$. Observe that the union of these core curves is homeomorphic to $S^1 \times \RR$. In fact, it is easy to see that this union is isometric to the standard $2$-dimensional cusp $\HH^2/\langle \zeta \mapsto \zeta + 1 \rangle$. \vs

\begin{figure*}
\centering{\includegraphics[width=\textwidth]{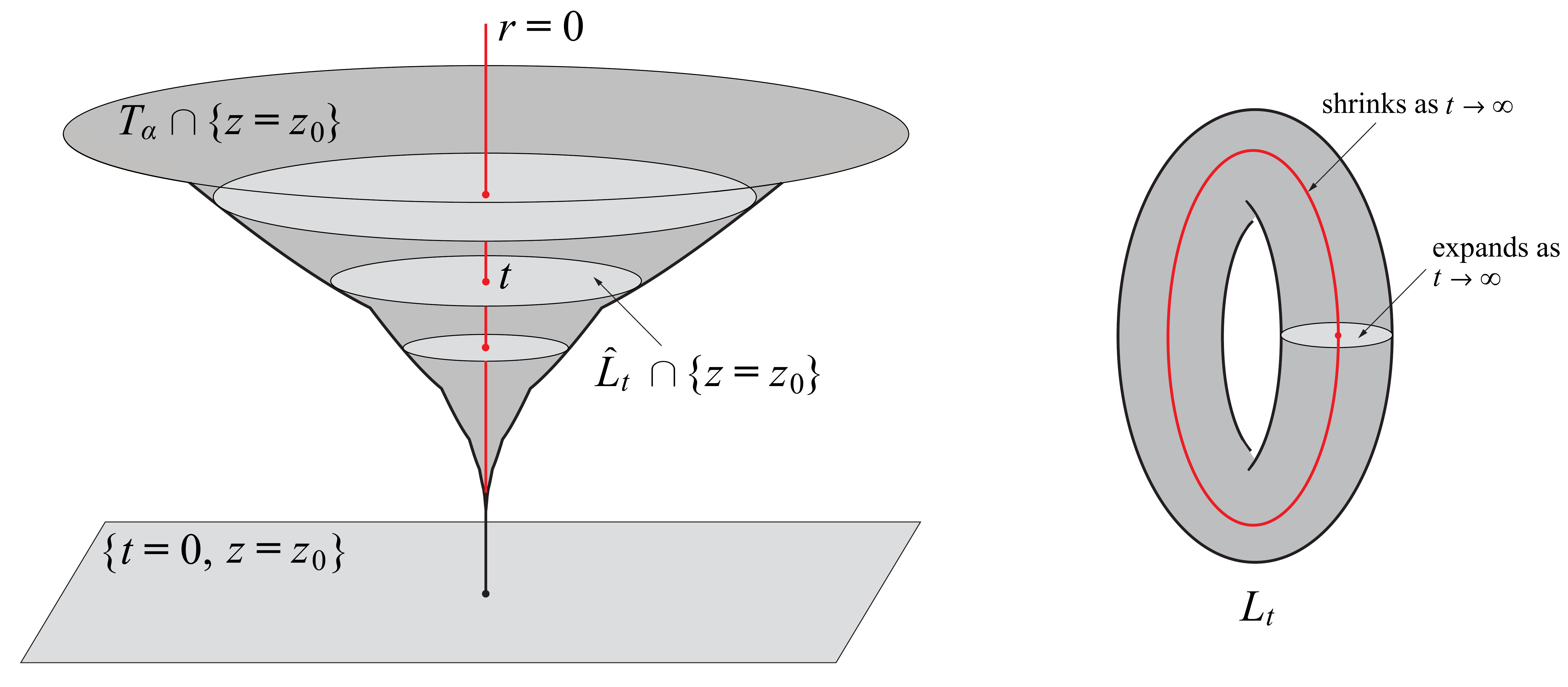}}
\caption{Left: A $z$-slice of the Margulis region $T_\alpha$ for an irrational $\alpha$. The horospheres based at $\infty$ intersect this slice along open $2$-disks, inducing a $g_\alpha$-invariant foliation of $T_\alpha$ into solid cylinders $\hat{L}_t$. Right: The leaf $L_t = \hat{L}_t/ \langle g_\alpha \rangle$ of the quotient foliation $\ff$ of the Margulis cusp $C_\alpha$, homeomorphic to a solid torus.}
\label{HOR} 
\end{figure*}

The foliation $\ff$ has an intrinsic description in terms of the geometry of the Margulis cusp $C_\alpha$. There is a unique $1$-dimensional foliation $\ff_G$ of $C_\alpha$ whose leaves consist of geodesics which forever stay in $C_\alpha$ in forward or backward time (they correspond to vertical geodesics in $T_\alpha$ landing at $\infty$), and $\ff$ is the $3$-dimensional foliation whose leaves are everywhere orthogonal to $\ff_G$. This allows us to recover the rotation angle $\alpha$ from the geometry of $C_\alpha$: Pick a leaf of $\ff$ and label it $L_1$. The leaf whose core curve is at distance $\log t$ from the core curve of $L_1$ will then be $L_t$. By \eqref{vol}, the function $t \mapsto \text{vol}(L_t)$ determines the inverse $\bb_\alpha^{-1}$, hence $\bb_\alpha$. The freedom in choosing the reference leaf $L_1$ means that this process only determines the conjugacy class $r \mapsto \lambda \bb_\alpha(\lambda^{-1}r)$ for $\lambda>0$, but that information is enough to determine $\alpha$ uniquely.

\subsection*{Bi-Lipschitz rigidity of Margulis cusps}

We now turn to the proof of Theorem B in \S \ref{sec:intro} on bi-Lipschitz rigidity of the Margulis cusps $C_\alpha$. Any two irrational screw translations $g_{\alpha}$ and $g_{\beta}$ are  topologically conjugate on their respective Margulis regions. For example, the map
\begin{equation}\label{psh}
\varphi:(r,\theta, z, t) \mapsto \big( r, \theta+(\beta-\alpha)z, z, t+\bb_{\beta}(r)-\bb_{\alpha}(r) \big)
\end{equation}
is easily seen to be a piecewise smooth homeomorphism $T_\alpha \to T_\beta$ which satisfies $\varphi \circ g_{\alpha}=g_{\beta} \circ \varphi$. But all such conjugacies must have unbounded geometry, as \thmref{a=b} below will show. \vs

Recall that a homeomorphism $\varphi: X \to Y$ between metric spaces is {\bit bi-Lipschitz} if there is a constant $k \geq 1$ such that
$$
k^{-1} d_X(x,\hat{x}) \leq d_Y(\varphi(x),\varphi(\hat{x})) \leq k \ d_X(x,\hat{x}) \qquad \text{for all} \ x,\hat{x} \in X.
$$
The smallest such $k$ is called the {\bit bi-Lipschitz constant} of $\varphi$. \vs

\begin{theorem}\label{a=b}
Suppose $\alpha, \beta \in \TT$ are irrational and $\varphi: T_\alpha \hookrightarrow \HH^4$ is a bi-Lipschitz embedding which satisfies $\varphi \circ g_{\alpha} = g_{\beta} \circ \varphi$ on $T_\alpha$. Then $\alpha=\beta$.
\end{theorem}

Notice that no a priori assumption is made on the image $\varphi(T_\alpha)$. This immediately gives Theorem B, for any bi-Lipschitz embedding $C_\alpha \hookrightarrow C_\beta$ lifts to a bi-Lipschitz embedding $T_\alpha \hookrightarrow T_\beta$ which conjugates $g_\alpha$ to $g_\beta$. As another application, we recover the following result of Kim in \cite{Ki}:

\begin{corollary}\label{kim}
The following conditions on irrational numbers $\alpha,\beta \in \TT$ are equivalent:
\begin{enumerate}
\item[(i)]
The maps $g_{\alpha} , g_{\beta} \in \isom$ are bi-Lipschitz conjugate.
\item[(ii)]
The restrictions $g_{\alpha}|_{\bd{\HH}^4}, g_{\beta}|_{\bd{\HH}^4} \in \mob(3)$ are quasiconformally conjugate.
\item[(iii)]
$\alpha=\beta$.
\end{enumerate}
\end{corollary}

\begin{proof}
The equivalence (i) $\Longleftrightarrow$ (ii) follows from a theorem of Tukia \cite{Tu}, while (i) $\Longrightarrow$ (iii) is immediate from \thmref{a=b}.
\end{proof}

The proof of \thmref{a=b} will be based on two lemmas. For the first lemma, consider the decomposition $T_\alpha =\bigcup_{j=1}^{\infty} T_{\alpha,j}$ of the Margulis region given by \eqref{oa}. Recall that
$$
T_{\alpha,j} = \{ x \in \HH^4: \rho(g_{\alpha}^j(x),x) < \ve \} =
\{ (r,\theta, z, t) : t > u_{\alpha,j}(r) \},
$$
where $u_{\alpha,j}$ is defined by \eqref{UJ}. Note that each $T_{\alpha,j}$ is invariant under $g_\alpha$.

\begin{lemma}\label{nbd}
Fix an integer $j \geq 1$ and consider a bi-Lipschitz embedding $\varphi: \bd T_{\alpha,j} \hookrightarrow \HH^4$ which satisfies $\varphi \circ g_{\alpha} = g_{\beta} \circ \varphi$. Then the image $\varphi(\bd T_{\alpha,j})$ lies in a neighborhood of $\bd T_{\beta,j}$ whose size depends on the choice of $\ve$ in $(0,\ve_4]$ and the bi-Lipschitz constant $k$ of $\varphi$. More precisely, if $x=(r,\theta,z,t) \in \bd T_{\alpha,j}$ and $y=\varphi(x)=(r',\theta',z',t')$, then
$$
K^{-1} u_{\beta,j}(r') \leq t' \leq K u_{\beta,j}(r')
$$
where $K=K(k,\ve) \geq 1$.
\end{lemma}

\begin{proof}
The condition $\rho(g_{\alpha}^j(x),x)=\ve$ implies $k^{-1} \ve \leq \rho(g_{\beta}^j(y),y) \leq k \ve$. The formulas \eqref{dist} for $\rho$ and \eqref{UJ} for $u_{\beta,j}$ show that
$$
\cosh \, \rho(g_{\beta}^j(y),y) -1 = \frac{1}{2} \left( \frac{u_{\beta,j}(r')}{c(\ve) t'} \right)^2,
$$
where, as before, $c(\ve)=1/\sqrt{2\cosh \ve -2}$. It follows that
\[
\frac{c(k \ve)}{c(\ve)} \leq \frac{t'}{u_{\beta,j}(r')} \leq \frac{c(k^{-1}\ve)}{c(\ve)}. \qedhere
\]
\end{proof}

\begin{remark}
The constant $K(k,\ve)$ is asymptotically independent of $\ve$ since the upper and lower bounds in the last inequality above tend to $k$ and $1/k$ as $\ve \to 0$.
\end{remark}

\begin{lemma}\label{ab}
Suppose $\alpha, \beta$ are irrationals, with $0<\beta<\alpha<1$. Then there is an increasing sequence $\{ n_i \}$ of positive integers such that $\lim_{i \to \infty} \| n_i \alpha \|=0$ and $\lim_{i \to \infty} \| n_i \beta \| \neq 0$.
\end{lemma}

\begin{proof}
Let $\Lambda$ be the closure of the additive subgroup of the torus $\TT^2=\TT \times \TT$ generated by $(\alpha,\beta)$. Clearly $\Lambda$ is infinite since $\alpha,\beta$ are irrational. \vs

If $\alpha,\beta$ are rationally independent, the classical theorem of Kronecker shows that $\Lambda=\TT^2$. In this case, there is an integer sequence $n_i \to \infty$ with $\| n_i \alpha \|$ tending to $0$ and $\| n_i \beta \|$ tending to any prescribed number in the interval $(0,1/2]$. \vs

If $\alpha,\beta$ are rationally dependent, then $\Lambda$ is a $1$-dimensional subgroup of $\TT^2$ homeomorphic to a circle. More precisely, suppose $m \alpha = n \beta$ for some positive integers $m,n$, where necessarily $1 \leq m < n$ since $\beta<\alpha$. Then $\Lambda$ is the image of the line $y=(m/n)x$ under the natural projection $\RR^2 \to \TT^2$, which wraps $n$ times horizontally and $m$ times vertically around $\TT^2$. In this case, there is an integer sequence $n_i \to \infty$ with $\| n_i \alpha \|$ tending to $0$ and $\| n_i \beta \|$ tending to any prescribed fraction of the form $j/n$ for $1 \leq j \leq n/2$.
\end{proof}

\begin{proof}[Proof of \thmref{a=b}]
Without loss of generality assume $\alpha,\beta$ are in $(0,1)$. First suppose $\beta<\alpha$. Let $\{ n_i \}$ be the sequence of positive integers given by \lemref{ab}, and choose an increasing sequence $\{ r_i \}$ of radii such that $\lim_{i \to \infty} n_i/r_i=0$. Define
$$
x_i = (r_i,0,0,u_{\alpha,1}(r_i)) \in \bd T_{\alpha,1} \qquad (i=1,2,3,\ldots)
$$
We have
$$
g_{\alpha}^{n_i}(x_i) = (r_i, n_i \alpha, n_i, u_{\alpha,1}(r_i)),
$$
hence
\begin{align*}
\cosh \rho (g_{\alpha}^{n_i}(x_i),x_i) -1 & = \frac{4 \sin^2(\pi n_i \alpha) \ r_i^2 + n_i^2}{2u_{\alpha,1}^2(r_i)} \\
& \asymp \frac{4 \sin^2 (\pi n_i \alpha) \ r_i^2 + n_i^2}{r_i^2} \\
& = 4 \sin^2(\pi n_i \alpha) + \left( \frac{n_i}{r_i} \right)^2.
\end{align*}
Since $\sin^2(\pi n_i \alpha) \asymp \| n_i \alpha \|^2 \to 0$ and $n_i/r_i \to 0$, it follows that
$$
\rho (g_{\alpha}^{n_i}(x_i),x_i) \to 0 \qquad \text{as} \ i \to \infty.
$$
If $y_i=\varphi(x_i)$, the bi-Lipschitz property of the conjugacy $\varphi$ implies that
$$
\cosh \rho (g_{\beta}^{n_i}(y_i),y_i) -1 \asymp \rho^2 (g_{\beta}^{n_i}(y_i),y_i) \to 0 \qquad \text{as} \ i \to \infty.
$$
Write $y_i=(r'_i,\theta'_i,z'_i,t'_i)$ so $g_{\beta}^{n_i}(y_i) = (r'_i, \theta'_i+n_i \beta, z'_i+n_i, t'_i)$, where $t'_i \asymp u_{\beta,1}(r'_i)$ by \lemref{nbd}. Hence
\begin{align*}
\cosh \rho (g_{\beta}^{n_i}(y_i),y_i) -1 & =  \frac{4 \sin^2 (\pi n_i \beta) \ {r'_i}^2 + {n_i}^2}{2{t'_i}^2} \\
& \asymp \frac{4 \sin^2 (\pi n_i \beta) \ {r'_i}^2 + {n_i}^2}{u_{\beta,1}^2(r'_i)} \\
& \asymp \frac{4 \sin^2 (\pi n_i \beta) \ {r'_i}^2 + {n_i}^2}{{r'_i}^2} \\
& \geq 4 \sin^2(\pi n_i \beta).
\end{align*}
It follows that $\| n_i \beta \| \asymp \sin(\pi n_i \beta) \to 0$ as $i \to \infty$, which is a contradiction. \vs

Next, suppose $\alpha<\beta$. Find a positive integer $n$ such that the fractional parts $\alpha'$ of $n \alpha$ and $\beta'$ of $n \beta$ satisfy $\beta' < \alpha'$. The isometry $\gamma: (r,\theta,z,t) \mapsto (r/n,\theta,z/n,t/n)$ conjugates the iterate $g_\alpha^n$ to $g_{\alpha'}$:
$$
\gamma \circ g_\alpha^n = g_{\alpha'} \circ \gamma \qquad \text{in} \ \HH^4.
$$
Using the definition of the Margulis region and the relation \eqref{inv}, we easily obtain the inclusion $T_{\alpha'} \subset \gamma(T_{\alpha})$. Thus, the restriction of the conjugate map $\gamma \circ \varphi \circ \gamma^{-1}$ is a bi-Lipschitz embedding $T_{\alpha'} \hookrightarrow \HH^4$ conjugating $g_{\alpha'}$ to $g_{\beta'}$, which is impossible by the first case treated above. We conclude that $\alpha=\beta$.
\end{proof}

\section{Appendix: On arithmetical characterization of presence}

The problem we investigate here is when a given denominator $q_n$ in the continued fraction expansion of an irrational number $\alpha$ is present in the boundary function $\bb_\alpha$ (see \S \ref{sec:ca}). We need only consider the case where $a_{n+1}=1$ since \corref{sufi} guarantees that $q_n$ is present when $a_{n+1} \geq 2$. Assuming $a_{n+1}=1$, the same corollary and the definition of fair triples show that
$$
q_{n} \ \text{is present} \quad \Longleftrightarrow \quad r_{n-1,n}<r_{n,n+1}.
$$
By the formula \eqref{rr1}, this condition can be written as
\begin{equation}\label{cc1}
q_{n} \ \text{is present} \quad \Longleftrightarrow \quad \frac{\delta_{n}-\delta_{n+1}}{\delta_{n-1}-\delta_{n}} < \frac{q_{n+1}^2-q_{n}^2}{q_{n}^2-q_{n-1}^2}.
\end{equation}
The right side of the inequality in \eqref{cc1} is easily computed:
$$
\frac{q_{n+1}^2-q_{n}^2}{q_{n}^2-q_{n-1}^2} = \frac{(q_{n}+q_{n-1})^2-q_{n}^2}{q_{n}^2-q_{n-1}^2} = \frac{2q_{n} q_{n-1} + q_{n-1}^2 }{q_{n}^2-q_{n-1}^2} = \frac{2 \mu +1}{\mu^2-1},
$$
where
\begin{equation}\label{MEW}
a_{n} < \mu = \frac{q_{n}}{q_{n-1}} < a_{n}+1.
\end{equation}
To estimate the left side of the inequality in \eqref{cc1}, we use the inequalities
$$
0.95 x^2 \leq \sin^2 x \leq x^2 \qquad \text{for} \ |x| \leq \frac{\pi}{12}
$$
which can be easily proved using calculus. Since the denominator $q_6$ is always $\geq 13$, by \lemref{qnp}(ii),
$$
\pi \| q_n \alpha \| <  \frac{\pi}{q_{n+1}} \leq \frac{\pi}{12} \qquad (n \geq 5).
$$
It follows that
$$
3.8 \pi^2 \| q_n \alpha \|^2 \leq \delta_n = 4 \sin^2 (\pi \| q_n \alpha \|) \leq 4 \pi^2 \| q_n \alpha \|^2 \qquad (n \geq 5).
$$
Introduce the quantity
$$
\lambda = \frac{\| q_{n} \alpha \|}{\| q_{n+1} \alpha \|}
$$
which by \lemref{qnp}(iii) satisfies
\begin{equation}\label{LAM}
a_{n+2} < \lambda < a_{n+2}+1.
\end{equation}
Note that since $a_{n+1}=1$, we have $\| q_{n-1} \alpha \| = \| q_{n} \alpha \| + \| q_{n+1} \alpha \|$, which shows
$$
\frac{\| q_{n-1} \alpha \|}{\| q_{n} \alpha \|} = 1 + \lambda^{-1}.
$$
Thus, for $n \geq 5$,
\begin{align*}
\frac{\delta_{n}-\delta_{n+1}}{\delta_{n-1}-\delta_{n}} & < \frac{4 \| q_{n} \alpha \|^2 - 3.8 \| q_{n+1} \alpha \|^2}{3.8 \| q_{n-1} \alpha \|^2 - 4 \| q_{n} \alpha \|^2} \\
& = \frac{1-0.95 \lambda^{-2}}{0.95(1+\lambda^{-1})^2-1} = \frac{\lambda^2-0.95}{-0.05 \lambda^2+1.9 \lambda + 0.95}.
\end{align*}
and
\begin{align*}
\frac{\delta_{n}-\delta_{n+1}}{\delta_{n-1}-\delta_{n}} & > \frac{3.8 \| q_{n} \alpha \|^2 - 4 \| q_{n+1} \alpha \|^2}{4 \| q_{n-1} \alpha \|^2 - 3.8 \| q_{n} \alpha \|^2} \\
& = \frac{0.95 - \lambda^{-2}}{(1+\lambda^{-1})^2-0.95} = \frac{0.95 \lambda^2-1}{0.05 \lambda^2+2 \lambda + 1}.
\end{align*}
Introducing the rational functions
\begin{align*}
X(t) & =\frac{2t +1}{t^2-1} \\
Y(t) & =\frac{t^2-0.95}{-0.05 t^2+1.9 t + 0.95} \\
Z(t) & =\frac{0.95 t^2-1}{0.05 t^2+2 t + 1},
\end{align*}
the condition \eqref{cc1} and the above estimates can be summarized as
\begin{equation}\label{cc2}
q_{n} \ \text{is present} \quad \Longleftrightarrow \quad \frac{\delta_{n}-\delta_{n+1}}{\delta_{n-1}-\delta_{n}} < X(\mu)
\end{equation}
and
\begin{equation}\label{cc3}
Z(\lambda) < \frac{\delta_{n}-\delta_{n+1}}{\delta_{n-1}-\delta_{n}} < Y(\lambda) \qquad (n \geq 5),
\end{equation}
where $\mu, \lambda$ satisfy \eqref{MEW} and \eqref{LAM}. \vs

\begin{figure*}
\centering{\includegraphics[width=0.5\textwidth]{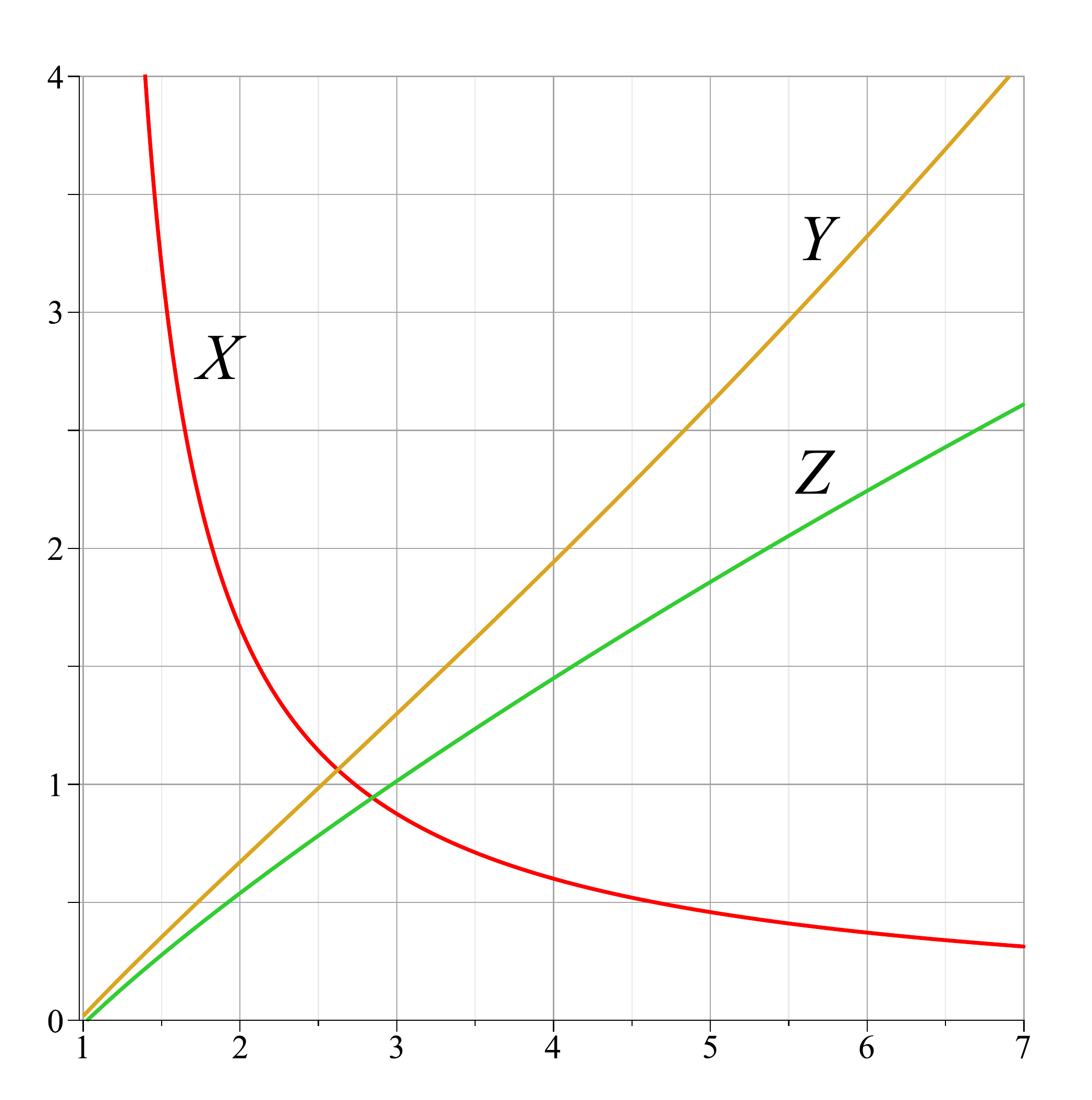}}
\caption{Graphs of the rational functions $X$, $Y$, and $Z$. Note that $Y$ has a singularity at $t \approx 38.5$ (not shown here) but that does not interfere with our estimates on the interval $[1,3]$.}
\label{RLU} 
\end{figure*}

The following can be deduced from \eqref{cc2} and \eqref{cc3} (compare the graphs of $X,Y,Z$ in \figref{RLU}): \vs

\begin{enumerate}
\item[$\bullet$]
If $a_{n} \geq 3$ and $a_{n+2} \geq 3$, then $\mu, \lambda>3$ and
$$
\frac{\delta_{n}-\delta_{n+1}}{\delta_{n-1}-\delta_{n}} > Z(\lambda)>Z(3)>X(3)>X(\mu).
$$
so $q_{n}$ is absent. \vs

\item[$\bullet$]
If $a_{n}=2$ and $a_{n+2} \geq 5$, then $2<\mu<3$, $\lambda>5$ and
$$
\frac{\delta_{n}-\delta_{n+1}}{\delta_{n-1}-\delta_{n}} > Z(\lambda)>Z(5)>X(2)>X(\mu),
$$
so $q_{n}$ is absent. \vs

\item[$\bullet$]
If $a_{n} \geq 5$ and $a_{n+2}= 2$, then $\mu >5$, $2<\lambda<3$ and
$$
\frac{\delta_{n}-\delta_{n+1}}{\delta_{n-1}-\delta_{n}} > Z(\lambda)>Z(2)>X(5)>X(\mu),
$$
so $q_{n}$ is absent. \vs

\item[$\bullet$]
If $a_{n}=1$ and $a_{n+2} \leq 2$, then $1<\mu<2$, $1<\lambda<3$ and
$$
\frac{\delta_{n}-\delta_{n+1}}{\delta_{n-1}-\delta_{n}} < Y(\lambda)<Y(3)<X(2)<X(\mu),
$$
so $q_{n}$ is present. \vs

\item[$\bullet$]
Finally, if $a_{n}=2$ and $a_{n+2}=1$, then $2<\mu<3$, $1<\lambda<2$ and
$$
\frac{\delta_{n}-\delta_{n+1}}{\delta_{n-1}-\delta_{n}} < Y(\lambda)<Y(2)<X(3)<X(\mu),
$$
so $q_{n}$ is present. \vs
\end{enumerate}

\begin{figure*}
\centering{\includegraphics[width=0.6\textwidth]{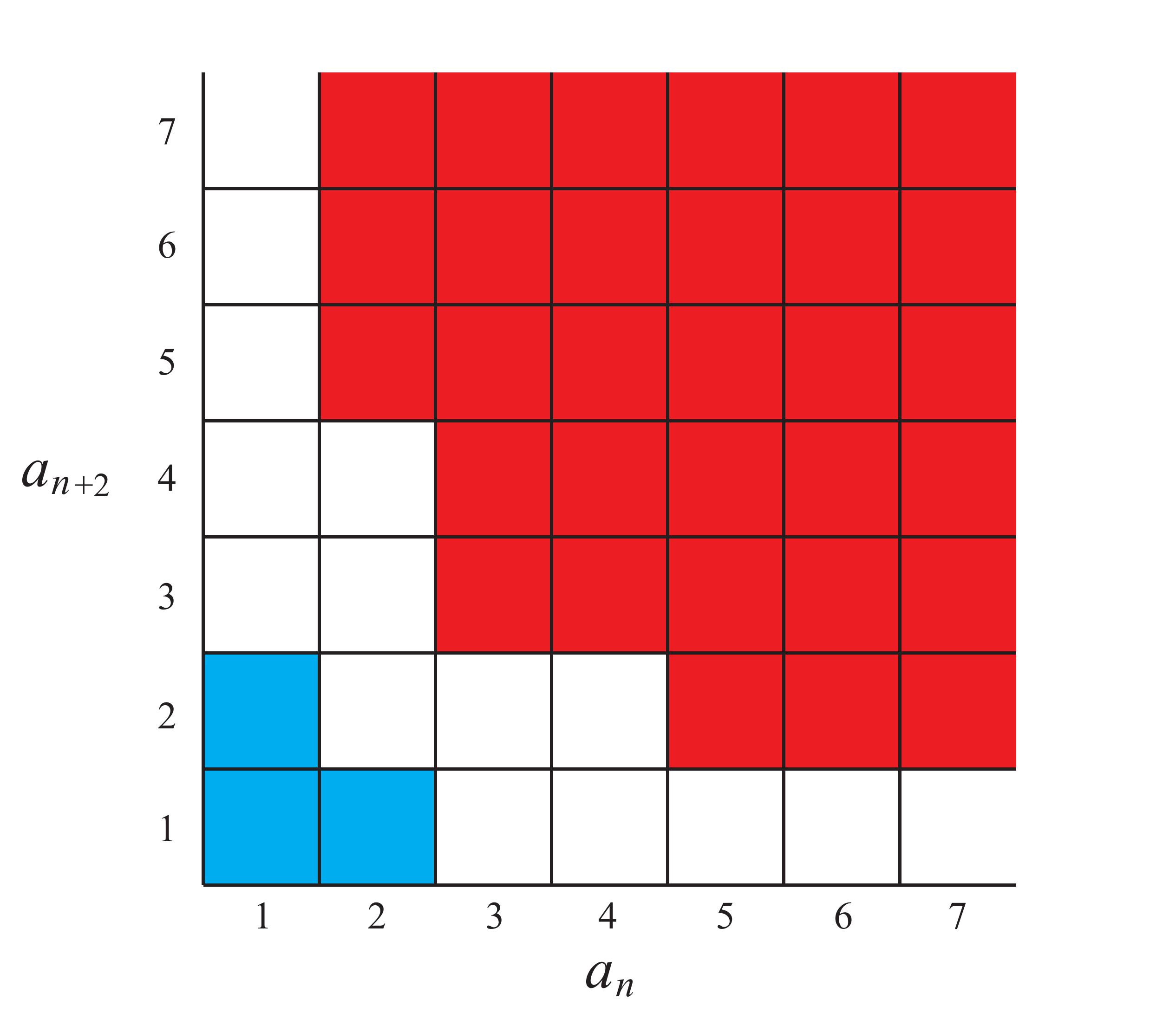}}
\caption{The locus of presence (blue) and absence (red) of $q_n$ in the $(a_{n},a_{n+2})$-plane when $a_{n+1}=1$. Here we assume $n \geq 5$. The white cells can go either blue or red depending on other partial quotients.}
\label{TAB} 
\end{figure*}

These findings are summarized in \figref{TAB}. In all other cases, the presence or absence of $q_n$ also depends on other
partial quotients such as $a_{n-1}$, $a_{n+3}$, etc.

\end{document}